\documentclass[11pt, reqno]{amsart}

\usepackage[T1]{fontenc}
\usepackage{xcolor}
\usepackage{latexsym}
\usepackage{amssymb}
\usepackage{mathrsfs}
\usepackage{amsmath}
\usepackage{fancybox,color}
\usepackage{enumerate}
\usepackage[latin1]{inputenc}
\usepackage{mathtools}

\usepackage{eurosym}
\usepackage{url}
\usepackage{hyperref}
  \hypersetup{colorlinks=true, %
   citecolor=magenta,%
   linkcolor=magenta}

\newcommand{\kom}[1]{}
\renewcommand{\kom}[1]{{\bf [#1]}}

 \def\1{\raisebox{2pt}{\rm{$\chi$}}}

\newtheorem{theorem}{Theorem}[section]

\newtheorem{lemma}[theorem]{Lemma}
\newtheorem{proposition}[theorem]{Proposition}
\newtheorem{definition}[theorem]{Definition}

\newcommand\blfootnote[1]{%
  \begingroup
  \renewcommand\thefootnote{}\footnote{#1}%
  \addtocounter{footnote}{-1}%
  \endgroup
}

 \def\1{\raisebox{2pt}{\rm{$\chi$}}}
 
\def\vint_#1{\mathchoice%
          {\mathop{\kern 0.2em\vrule width 0.6em height 0.69678ex depth -0.58065ex
                  \kern -0.8em \intop}\nolimits_{\kern -0.4em#1}}%
          {\mathop{\kern 0.1em\vrule width 0.5em height 0.69678ex depth -0.60387ex
                  \kern -0.6em \intop}\nolimits_{#1}}%
          {\mathop{\kern 0.1em\vrule width 0.5em height 0.69678ex
              depth -0.60387ex
                  \kern -0.6em \intop}\nolimits_{#1}}%
          {\mathop{\kern 0.1em\vrule width 0.5em height 0.69678ex depth -0.60387ex
                  \kern -0.6em \intop}\nolimits_{#1}}}
\def\vintslides_#1{\mathchoice%
          {\mathop{\kern 0.1em\vrule width 0.5em height 0.697ex depth -0.581ex
                  \kern -0.6em \intop}\nolimits_{\kern -0.4em#1}}%
          {\mathop{\kern 0.1em\vrule width 0.3em height 0.697ex depth -0.604ex
                  \kern -0.4em \intop}\nolimits_{#1}}%
          {\mathop{\kern 0.1em\vrule width 0.3em height 0.697ex depth -0.604ex
                  \kern -0.4em \intop}\nolimits_{#1}}%
          {\mathop{\kern 0.1em\vrule width 0.3em height 0.697ex depth -0.604ex
                  \kern -0.4em \intop}\nolimits_{#1}}}

\newcommand{\kint}{\vint}

\newcommand{\aveint}[2]{\mathchoice%
          {\mathop{\kern 0.2em\vrule width 0.6em height 0.69678ex depth -0.58065ex
                  \kern -0.8em \intop}\nolimits_{\kern -0.45em#1}^{#2}}%
          {\mathop{\kern 0.1em\vrule width 0.5em height 0.69678ex depth -0.60387ex
                  \kern -0.6em \intop}\nolimits_{#1}^{#2}}%
          {\mathop{\kern 0.1em\vrule width 0.5em height 0.69678ex depth -0.60387ex
                  \kern -0.6em \intop}\nolimits_{#1}^{#2}}%
          {\mathop{\kern 0.1em\vrule width 0.5em height 0.69678ex depth -0.60387ex
                  \kern -0.6em \intop}\nolimits_{#1}^{#2}}}

\numberwithin{equation}{section}
\allowdisplaybreaks

\definecolor{color1}{rgb}{0.309, 0.43,0.258}
\definecolor{color2}{rgb}{0.741, 0.502,0.743}
\definecolor{color3}{rgb}{0.580, 0.163,0.107}
\definecolor{color0}{rgb}{0.785, 0.333, 0.654}
\begin{document}

\title[]{Regularity for fully nonlinear elliptic equations in generalized Orlicz spaces}

\author[Byun]{Sun-Sig Byun}
\address{Department of Mathematical Sciences, Seoul National University,
Seoul 08826, Republic of Korea}
\email{byun@snu.ac.kr}

\author[Han]{Jeongmin Han}
\address{Department of Mathematics, Soongsil University,
Seoul 06978, Republic of Korea}
\email{jeongmin.han@ssu.ac.kr}

\author[Lee]{Mikyoung Lee}
\address{Department of Mathematics and Institute of Mathematical Science, Pusan National University,
Busan 46241, Republic of Korea}
\email{mikyounglee@pusan.ac.kr}

\blfootnote{S.-S. Byun was supported by the National Research Foundation of Korea grant (NRF-2022R1A2C1009312) funded by the Korean Government. J. Han was supported by the National Research Foundation of Korea grant (NRF-2019R1F1A1061295) funded by the Korean Government.}


\begin{abstract}
In this paper, we establish an optimal global Calder\'{o}n-Zygmund type estimate for the viscosity
solution to the Dirichlet boundary problem of fully nonlinear elliptic equations with possibly nonconvex nonlinearities. 
We prove that the Hessian of the solution is as integrable as the nonhomogeneous term in the
setting of a given generalized Orlicz space even when the
nonlinearity is asymptotically convex with respect to the Hessian of
the solution.
\end{abstract}

\maketitle


\section{Introduction}

This paper studies the following fully nonlinear elliptic problem
with the zero boundary condition
\begin{align}\label{ob_acori}
\left\{ \begin{array}{ll}
F(D^{2}u,Du,u,x) = f & \textrm{in $\Omega, $}\\
u=0 & \textrm{on $\partial \Omega$,}
\end{array} \right.
\end{align}
where $\Omega \subset \mathbb{R}^{n}$ is a $C^{1,1}$-domain. Here we
aim at deriving an optimal global Calder\'{o}n-Zygmund theory for
the problem \eqref{ob_acori} when the nonlinearity $F=F(X,\xi,z,x)$
is not necessarily convex with respect to the $X$ variable for the
Hessian $D^2u(x)$ of the solution by identifying minimal conditions
to impose on the nonlinearity in the setting of a given generalized
Orlicz space. More precisely speaking, we ask what are reasonable
requirements on the nonlinearity $F=F(X,\xi,z,x)$, the boundary of
the domain $\partial \Omega$ and the generalized Orlicz function
$\varphi=\varphi(t,x)$ under which the following implication
$$ \varphi(x, |f(x)|) \in L^1(\Omega) \Longrightarrow \varphi(x,
|D^2u(x)|) \in L^1(\Omega)$$ holds for every $\varphi$. This is
so-called an optimal global $W^{2, \varphi(\cdot)}$-regularity
theory, as it can be reduced to a global $W^{2,p}$-regularity for
the classical case when $\varphi(x,t)=t^p$ for $1<p<\infty$.
$W^{2,p}$-regularity is the fundamental type of regularity in
Calder\'{o}n-Zygmund theory. For this, we refer to \cite{MR1088476,
MR1191890, MR3335124, MR3479199} and references therein. 
In \cite{MR4046185}, the authors studied $W^{2,p}$-regularity for oblique derivative boundary problems when
$F$ is convex with respect to $X$-variable under the setting of the
classical Lebesgue spaces. The present work can be regarded as a
natural outgrowth and improvement of the one in \cite{MR3957152, MR4258789}, as
we allow $F$ to be merely nonconvex in $X$-variable provided the
Hessian of a solution is under control in the asymptotic sense while
involving a much wider class of function spaces with highly
nonstandard growth including Orlicz growth, variable exponent growth
and double phase growth. In the next section we will give a precise
description of the underlying function space and the detailed
assumptions on the nonlinearity to state the main result.

To establish the regularity of the Hessian of solutions for \eqref{ob_acori}, we
observe $C^{1,1}$-regularity of the corresponding reference problem
as we are able to use the convexity of the  limiting nonlinearity
and localize the problem near the flat boundary in order to apply
the regularity previously established in, for instances,
\cite{MR1351007, MR3780142}). Indeed Calder\'{o}n-Zygmund type
estimates for fully nonlinear equations have been obtained under a
certainly weakened convexity condition for $F$ in a more restricted
problem in the recent papers including \cite{MR3040680, MR3479199,
MR3539473, MR3957152, MR4258789}. The approach used there is based
on the observation that the higher integrability of the Hessian of a
solution is connected to the case when $|D^{2}u|$ is sufficiently
large. Meanwhile we consider a solution of \eqref{ob_acori} in
generalized Orlicz spaces. A generalized Orlicz space can be
understood as an Orlicz space with the variable exponent growth.
This function space includes not only fundamental spaces like
$L^{p}$-space, but also a class of functions having nonstandard
growth such as double phase phenomena. One can find various and
interesting examples in \cite{MR3931352}. The regularity theory of
partial differential equations has developed a lot to the setting of
this function space, as we can find very interesting papers
\cite{MR3165506, MR3418078, MR3811530, MR3592600, MR3360438,
MR3732884}. The classical method via Calder\'{o}n-Zygmund
decomposition and maximal function has been used to obtain
regularity results under various settings. We refer to
\cite{MR3900364} in variable exponent spaces and \cite{MR3957152} in
Orlicz spaces, respectively. On the other hand it seems very
difficult to employ this argument to the function space having both
variable exponent growth and Orlicz growth at the same time as in
our case. To overcome this difficulty, we use the approach
introduced by Acerbi and Mingione in \cite{MR2286632}. This
methodology employs a stopping time argument (see \cite{MR2352517}
for its origin) and a Vitali type of covering lemma. The key
ingredient is an approximation lemma, Lemma \ref{thetaest}, which
gives a $W^{2,\nu}$-type estimate for the difference between a
solution of the original problem and that of its limiting one. The
advantage of this approach enables us to consider more general
function spaces, since belonging to class $S$ (see Section 2) is the
only needed condition to obtain $W^{2,\nu}$-type estimates.

This paper is organized as follows. In Section 2 we introduce basic
notions and background knowledge to state the main result, Theorem
\ref{gor_main}. The main part of this paper is in Section 3 in which
we show boundary Hessian regularity results. Section 4 is devoted to
presenting the gradient estimate and proving the main theorem.

\section{Preliminaries}

\subsection{Notations}
\begin{itemize}
\item For $x= (x_{1}, \dots, x_{n-1},x_{n} ) \in \mathbb{R}^{n} $, we write $x'=(x_{1},  \dots, x_{n-1} )$.
\item $Sym(n)$ is the set of $n \times n $ symmetric matrices and $||X||=\sup_{|x| \le 1} |Xx|$ for every $ X\in Sym(n)$.
\item $\mathbb{R}_{+}^{n} : = \{ x \in \mathbb{R}^{n}: x_{n} > 0 \} $.
\item For $x_{0} \in  \mathbb{R}^{n} $ and $r>0$, $B_{r}(x_{0}):=\{ x \in  \mathbb{R}^{n} :|x-x_{0}|<r \}  $. We write $B_{r}=B_{r}(0)$, $B_{r}^{+}=B_{r} \cap \mathbb{R}_{+}^{n}$ and $B_{r}^{+}(x_{0})=B_{r}(x_{0}) \cap \mathbb{R}_{+}^{n}$.
\item We write $T_{r} = B_{r} \cap \{ x \in \mathbb{R}^{n} : x_{n}=0 \} $ and $T_{r}(x_{0}') := T_{r} + x_{0}' $ for $ x_{0}' \in \mathbb{R}^{n-1} $.
\item For any measurable set $A$ with $|A| \neq 0 $ and measurable function $f$, we write
$$ \kint_{A} f dx = \frac{1}{|A|} \int_{A} f dx. $$
\item For a locally integrable function $f:\mathbb{R}^n\to\mathbb{R}$, the (Hardy-Littlewood) maximal operator $\mathcal{M}$ is defined by 
$$\mathcal{M}(f)(x)=\sup_{r>0}\kint_{B_r(x)}|f(y)|dy $$
for $x\in \mathbb{R}$.
\end{itemize}
\subsection{Backgrounds}
In this subsection, we present some definitions and basic concepts to be used throughout this paper.

We first introduce a generalized Orlicz space (or Musielak-Orlicz space), the main function space we are considering.
To do this, we now recall here the notion of a weak $\Phi$-function.
A function $\varphi:[0, \infty) \times \Omega \to [0,\infty)$ is called a weak $\Phi$-function
if $\varphi=\varphi (t,x) $ is measurable in the $x$-variable, and increasing and left-continuous in the $t$-variable,
and satisfies that
$\varphi(0,x)=0$,
$\lim_{t \to \infty} \varphi (t,x)= \infty $
and the condition $\textrm{(aInc)}_{1}$ (see below).

We further need some structure assumptions on $\varphi$ to derive our regularity results for the problem \eqref{ob_acori} as we state now.
We write
$$\varphi_{U}^{+}(t)=\sup_{x\in U}\varphi(t,x) \quad \textrm{and} \quad \varphi_{U}^{-}(t)=\inf_{x\in U}\varphi(t,x) $$
for  any $U \subset \Omega$. For convenience, we write $ \varphi_{\Omega}^{\pm}=\varphi^{\pm}$.
The following conditions are assumed throughout this paper.
\begin{itemize}
\item[(A0)] There exists $t_{0} \in (0,1)$ such that $\varphi^{+}(t_{0}) \le 1 \le \varphi^{-}(1/t_{0})$.
\item[$\textrm{(aInc)}_{p}$] There exists $L \ge 1$ such that the map $(0,\infty) \ni t \mapsto \varphi (t,x)/ t^{p}$ is $L$-almost increasing  (i.e., $\frac{\varphi(t,x)}{t^p}\le L\frac{\varphi(s,x)}{s^p}$ for $0<t\le s$) for every $x\ \in \Omega$.
\item[$\textrm{(aDec)}_{q}$] There exists $L \ge 1$ such that the map $(0,\infty) \ni t \mapsto \varphi (t,x) / t^{q}$ is $L$-almost decreasing (i.e., $\frac{\varphi(t,x)}{t^q}\ge L\frac{\varphi(s,x)}{s^q}$ for $0<t\le s$) for every $x\ \in \Omega$.
\item[(A1-$\varphi^{-}$)] There exists $t_{0} \in (0,1)$ such that for every ball $B \subset \Omega$
$$\varphi_{B}^{+}(t_{0} t) \le \varphi_{B}^{-}(t)  \quad \textrm{for any} \ t \in [1, (\varphi_{\Omega}^{-})^{-1}(|B|^{-1})] .$$
\end{itemize}

We next return to function spaces associated with a weak $\Phi$-function $\varphi$ satisfying the structure conditions mentioned above.
Let $L^{0}(\Omega)$ be the space of measurable functions $v: \Omega \to \mathbb{R}$.
We define
$$L^{\varphi(\cdot)}(\Omega)=\{ v \in L^{0}(\Omega):  \lim_{\lambda \to 0} \rho_{\varphi(\cdot)}(\lambda v)=0  \}, $$
where $\rho_{\varphi(\cdot)}$ is defined by
$$\rho_{\varphi(\cdot)}( v)= \int_{\Omega}\varphi(|v|,x) dx$$
with
$$||v||_{L^{\varphi(\cdot)}(\Omega)}=\inf \{ \lambda >0: \rho_{\varphi(\cdot)}(v/\lambda ) \le 1 \} .$$
If $\varphi$ is independent of the $x$-variable, we write $L^{\varphi(\cdot)}(\Omega)=L^{\varphi}(\Omega)$.


The generalized Orlicz-Sobolev space $W^{2,\varphi(\cdot)}(\Omega)$ is defined by
$$ W^{2,\varphi(\cdot)}(\Omega)=\big\{ v\in W^{2,1}(\Omega) : \sum_{0\le|\alpha|\le 2}||v ||_{W^{2,\varphi(\cdot)}(\Omega)} :=||D^{\alpha}v||_{L^{\varphi(\cdot)}(\Omega)}< \infty \big\},$$
where
$ W^{2,1}(\Omega) $ is the usual standard Sobolev space.

We will use a Jensen type inequality to obtain our estimates. The proof can be found in \cite[Lemma 2.3]{MR3953020}.
\begin{lemma} \label{orljen}
Let $\varphi$ be a weak $\Phi$-function satisfying $\textrm{(aInc)}_{p}$ for some $p \ge 1$.
Then there exists $C=C(p,L)>0$ such that
$$ \varphi \bigg(C \bigg(\kint_{\Omega}|f|^{p} dx \bigg)^{\frac{1}{p}} \bigg)  \le \kint_{\Omega} \varphi(|f|)dx.$$
\end{lemma}

The following proposition gives some properties of the inverse function of a weak $\Phi$-function (see \cite[Proposition 2.3.13]{MR3931352}).
\begin{proposition}\label{phiftninv}If $\varphi$ is a weak $\Phi$-function, then $\varphi^{-1}$ is increasing, left-continuous,
and satisfies $\textrm{(aDec)}_{1}$ and
$$\varphi^{-1}(0,x)=0  \quad \textrm{and} \quad
\lim_{t \to \infty} \varphi^{-1} (t,x)= \infty .$$
\end{proposition}

Next, we introduce a weak concept of second order derivatives. This notion plays an important role in establishing our Hessian estimates for the problem \eqref{ob_acori}.

Let $V \subset \Omega $, $M>0 $ and $u \in C(\Omega) $.
We define
\begin{align*}
 \underline{\mathcal{G}}_{M}(u, V) \hspace{-0.2em}=\hspace{-0.2em} \left\{ \hspace{-0.2em} x_{0} \in V : \hspace{-0.2em}  \begin{array}{ll}
 \textrm{there is a concave paraboloid $P$ with $|D^{2}P|=M$}\\
\textrm{such that $P(x_{0})=u(x_{0})$ and}\\
\textrm{$P(x)\le u(x)$ for any $x \in V$}
\end{array} \hspace{-0.2em}  \right\}
\end{align*}and
\begin{align*}
\underline{\mathcal{A}}_{M}(u, V) = V \backslash \underline{\mathcal{G}}_{M}(u, V) .
\end{align*}
Likewise, $ \overline{\mathcal{G}}_{M}(u, V)$ and $\overline{\mathcal{A}}_{M}(u, V) $ can be defined with convex paraboloids.
And we set
\begin{align*}&\mathcal{G}_{M}(u,V) =  \underline{\mathcal{G}}_{M}(u, V)  \cap \overline{\mathcal{G}}_{M}(u, V), \\
& \mathcal{A}_{M}(u,V)  =  \underline{\mathcal{A}}_{M}(u, V)  \cup \overline{\mathcal{A}}_{M}(u, V).
\end{align*}
Then we consider the following functions
\begin{align*} &\underline{\Theta}(u,V)(x) = \inf \{ M>0 :x \in \underline{\mathcal{G}}_{M}(V) \},
\\ & \overline{\Theta}(u,V)(x) = \inf \{ M>0 :x \in \overline{\mathcal{G}}_{M}(V) \},
\end{align*}
and
$$\Theta (u,V)(x) = \sup \{ \underline{\Theta}(u,V)(x) ,\overline{\Theta}(u,V)(x) \}. $$

We finish this subsection with introducing Pucci extremal operator and the class $S$.
\begin{definition}[Pucci extremal operator]
For any $X \in Sym(n) $, the Pucci extremal operators $ \mathcal{M}^{+} $ and  $\mathcal{M}^{-}$ are defined as
$$  \mathcal{M}^{+}(\lambda, \Lambda, X)=\Lambda \sum_{e_{i}>0} e_{i} + \lambda \sum_{e_{i}<0} e_{i} \ \textrm{and}\  \mathcal{M}^{-}(\lambda, \Lambda, X)=\lambda \sum_{e_{i}>0} e_{i} + \Lambda \sum_{e_{i}<0} e_{i} $$
where $ e_{i}$ ($1 \le i \le n $) are eigenvalues of $X$.
\end{definition}

\begin{definition}
Let $0 < \lambda \le \Lambda $ and $f \in L^{n}(\Omega) $.
We define the classes $\underline{S}(\lambda, \Lambda, f) $ ($ \overline{S}(\lambda, \Lambda,  f), respectively$) to be the set of all continuous functions $u$ that satisfy $ \mathcal{M}^{+}u \ge f$($ \mathcal{M}^{-}u \le f$, respectively) in the viscosity sense (see Definition \ref{defvs}) in $\Omega$.
We also write
$$S(\lambda, \Lambda,  f) =\overline{S}(\lambda, \Lambda, f) \cap \underline{S} (\lambda, \Lambda,  f)  $$
and
$$S^{\ast}(\lambda, \Lambda,  f) =\overline{S}(\lambda, \Lambda,| f|) \cap \underline{S} (\lambda, \Lambda,  -|f|) . $$
\end{definition}

\subsection{Main result}
Throughout this paper, we assume that $F=F(X,\xi,z,x)$ satisfies the following structure conditions:
there are some constants $\lambda, \Lambda, b,c>0$ with
$\lambda \le \Lambda$ such that
\begin{align} \label{ob_sc} \begin{split}
& \mathcal{M}^{-} (\lambda, \Lambda, X_{1}-X_{2}) -b|\xi_{1}-\xi_{2}| - c|z_{1}-z_{2}| \\ &
\qquad \le F(X_{1},\xi_{1},z_{1},x) - F(X_{2},\xi_{2},z_{2},x) \\ &
\qquad \qquad \le \mathcal{M}^{+} (\lambda, \Lambda, X_{1}-X_{2}) +b|\xi_{1}-\xi_{2}| + c|z_{1}-z_{2}|
\end{split}
\end{align}
for any $X_{1},X_{2} \in Sym(n)$, $\xi_{1},\xi_{2} \in \mathbb{R}^{n}$, $z_{1},z_{2} \in \mathbb{R}$ and $x \in \mathbb{R}^{n}$.
We further assume that
\begin{align} \label{foruni}
d(z_{2}-z_{1}) \le F(X,\xi,z_{1},x)-F(X,\xi,z_{2},x)
\end{align}
for some $d>0$ and any $X\in Sym(n)$, $\xi \in \mathbb{R}^{n}$, $z_{1},z_{2} \in \mathbb{R}$ with $z_{1}\le z_{2}$ and $x \in  \mathbb{R}^{n}$.

We introduce the notion of the viscosity solution here. 
\begin{definition}\label{defvs} Let $F=F(X,\xi,s,x)$ be continuous in $X,\xi,s$ and measurable in $x$.
Suppose that $f \in L^{n}(\Omega )\cap C(\overline{\Omega})$.
A continuous function $u $ is called a ($C$-)viscosity solution of \eqref{ob_acori} if the following conditions hold:
\begin{itemize}
\item[(a)] for all $ \varphi \in C^{2}(\Omega) $ touching $u$ by above at $x_{0} \in \Omega$, $$F( D^{2} \varphi (x_{0}), D \varphi (x_{0}),  u(x_{0}),  x_{0}) \ge f(x_{0}) ,$$ 
\item[(b)] for all $ \varphi \in C^{2}(\Omega) $ touching $u$ by below at $x_{0} \in \Omega$, $$F( D^{2} \varphi (x_{0}), D \varphi (x_{0}),  u(x_{0}) ,x_{0}) \le f(x_{0}) .$$ 
\end{itemize}
\end{definition}
Note that the assumptions \eqref{ob_sc} and \eqref{foruni} ensure the existence and uniqueness of a viscosity solution to the problem \eqref{ob_acori} (see \cite[Theorem 4.6]{MR2486925}).

The recession operator $F^{\star}$ associated with $F$ is defined by
\begin{align} \label{fstar}F^{\star}(X,\xi,z,x):=\lim_{\mu \to 0^{+}} F_{\mu}(X,\xi,z,x)
\end{align}
if it exists, where $ F_{\mu}(X,\xi,z,x):=  \mu F( \mu^{-1}X,\xi, z,x) $.



We define a function $\sigma_{F}$ for $F$ as follows:
$$ \sigma_{F}(x,x_{0}) := \sup_{X \in Sym(n)} \frac{|F(X,0,0,x)-F(X,0,0,x_{0})|}{||X||+1}$$
for any $x, x_{0} \in \mathbb{R}^{n}$.

Now we state our main result.
\begin{theorem}\label{gor_main} 
Let $\Omega \subset \mathbb{R}^n$ with the $C^{1,1}$-boundary and $F=F(X,\xi,z,x)$ satisfy \eqref{ob_sc} and \eqref{foruni}  with $F(0,0,0,\cdot)=0$.
Assume that there exists $F^{\star}$ satisfying \eqref{fstar} which is convex in $X$.
Moreover, let $\varphi$ be a weak $\Phi$-function satisfying (A0), (A1-$\varphi^{-}$), $\textrm{(aInc)}_{p}$ and $\textrm{(aDec)}_{q}$, and set $\psi(t,x)=\varphi(t^{n},x)$ for some $1<p\le q<\infty$.

Suppose $f \in  L^{\psi(\cdot)}(\Omega)$.
Then there exists a small $\delta=\delta(n,\lambda,\Lambda,p,q,L,t_{0})>0$ such that
if for any $r>0$ and $x_0\in \Omega$,
$$\bigg( \kint_{B_{r}(x_{0})\cap \Omega} \sigma_{F^{\star}}(x_{0},x)^{n} dx  \bigg)^{\frac{1}{n}} \le \delta, $$
then the unique viscosity solution $u $ to \eqref{ob_acori} satisfies
$u,Du,D^{2}u \in  L^{\psi(\cdot)}(\Omega)$
with the estimate
\begin{align}\label{est_main}
    ||u||_{ W^{2, \psi(\cdot)}( \Omega)} \le C||f||_{ L^{\psi(\cdot)}(\Omega)}
\end{align}
for some $C=C(n,\lambda,\Lambda,p,q,L,t_{0},b,c,\Omega)>0$.
\end{theorem}

\subsection{Auxiliary lemmas}

We provide some known regularity results and analytic tools to be used later here.

The following lemma enables us to use the function $\Theta$ to derive our results instead of $D^{2}u$.
We omit the proof here, but one can prove this lemma by using the reflexivity (\cite[Corollary 3.6.7]{MR3931352}) and the H\"{o}lder inequality
(\cite[Corollary 3.2.11]{MR3931352}) of generalized Orlicz spaces
(see \cite[Lemma 3.4]{MR3695962} for the proof of the original Orlicz case).

\begin{lemma} \label{eqthehe}
Assume that $\varphi$ is a weak $\Phi$-function satisfying $\textrm{(aInc)}_{p}$ and $\textrm{(aDec)}_{q}$ for some $1<p\le q<\infty$.
Let $r>0$ and $u$ be a continuous function in $\Omega$.
Set $\Theta(u,r)(x)=\Theta(u, \Omega \cap B_{r}(x))(x)$ for $x \in \Omega$.
Then if $\Theta (u,r) \in L^{\varphi(\cdot)}(\Omega)$, we have $D^{2}u \in  L^{\varphi(\cdot)}(\Omega)$ with
$$||D^{2}u||_{L^{\varphi(\cdot)}(\Omega)} \le 8||\Theta(u,r)||_{L^{\varphi(\cdot)}(\Omega)} .$$
\end{lemma}

We also mention a regularity estimate for the function $\Theta$.
This can be shown by using the results \cite[Lemma 2.9 - Proposition 2.12]{MR2486925} and the method in the proof of \cite[Proposition 7.4]{MR1351007}.

\begin{lemma} \label{w2del}
Let $0<\lambda \le \Lambda $.  
Then the following hold:
\begin{itemize}
    \item[(a)] Suppose $ u \in S(\lambda, \Lambda,  f) $ in $B_{1} $, $f \in L^{n}(B_{1})$
and $||u||_{L^{\infty}(B_{1})}\le 1$.Then  there exists a constant $\nu=\nu(n,\lambda,\Lambda)\in(0,1)$ such that
$\Theta = \Theta(u,B_{1/2}) \in L^{\nu}(B_{1/2})$ with the estimate
$$||\Theta||_{L^{\nu}(B_{1/2})} \le C(||u||_{L^{\infty}(B_{1})}+||f||_{L^{n}(B_{1})}  )$$
for some $C=C(n,\lambda,\Lambda)>0$.
    \item[(b)] Suppose $ u \in S(\lambda, \Lambda,  f) $ in $B_{1}^{+}$, $f \in L^{n}(B_{1}^{+})$ 
and $||u||_{L^{\infty}(B_{1}^{+})}\le 1$.Then  there exists a constant $\nu=\nu(n,\lambda,\Lambda)\in(0,1)$ such that
$\Theta = \Theta(u,B_{1/2}^{+}) \in L^{\nu}(B_{1/2}^{+})$ with the estimate
$$||\Theta||_{L^{\nu}(B_{1/2}^{+})} \le C(||u||_{L^{\infty}(B_{1}^{+})}+||f||_{L^{n}(B_{1}^{+})}  )$$
for some $C=C(n,\lambda,\Lambda)>0$.
\end{itemize}
\end{lemma}

We also introduce the following technical lemma which will be used later for the proof of main estimates (for the proof, see \cite[Lemma 4.3]{MR2777537}).
\begin{lemma} \label{tchlm}
Let $\phi : [1,2] \to \mathbb{R}$ be a nonnegative bounded function.
Suppose that for any $s_{1},s_{2}$ with $1\le s_{1}<s_{2}\le 2$,
$$ \phi(s_{1}) \le a \phi(s_{2})+\frac{C_{1}}{(s_{2}-s_{1})^{\gamma}}+C_{2} $$
where $C_{1},C_{2}>0$, $\gamma>0$ and $0\le a <1$.
Then we have
$$ \phi(s_{1}) \le c\bigg(\frac{C_{1}}{(s_{2}-s_{1})^{\gamma}}+C_{2}\bigg)$$
for some $c=c(\gamma,a)>0$.
\end{lemma}

\section{Hessian estimates}
To discuss the global regularity estimate for \eqref{ob_acori}, we first need to be concerned with local interior and boundary estimates. 
This section is mainly devoted to investigating Hessian estimates for the problem
 under the assumptions described in the previous section.

We recall the function
$$ \sigma_{F}(x,x_{0}) = \sup_{X \in Sym(n)} \frac{|F(X,x)-F(X,x_{0})|}{||X||+1}$$
for any $x, x_{0} \in \mathbb{R}^n$.
Throughout this section, 
We assume that $\varphi$ is a weak $\Phi$-function satisfying (A0), (A1-$\varphi^{-}$), $\textrm{(aInc)}_{p}$ and $\textrm{(aDec)}_{q}$, and set $\psi(t,x)=\varphi(t^{n},x)$.

We state our main result of this section here. 
This gives a Calder\'{o}n-Zygmund type estimate for solutions to \eqref{ob_muin} and \eqref{ob_mubdry}.
\begin{theorem}\label{main_ac} 
Assume that $F=F(X,x)$ is uniformly elliptic with constants $\lambda$ and $\Lambda$ with $F(0,\cdot)=0$,
and there exists $F^{\star}=F^{\star}(X,x)$ which is convex in $X$. Then the following hold:
\begin{itemize}
\item[(a)] Suppose $f \in L^{\psi(\cdot)}(B_{6})$ and let $u$ be a viscosity solution to
\begin{align}
\label{ob_muin}
F(D^{2}u,x) = f \quad \textrm{in} \ B_{6}.
\end{align}
Then there exists $\delta=\delta(n,\lambda,\Lambda,p,q,L,t_{0})$ such that
if
$$\bigg( \kint_{B_{r}(x_{0})\cap B_1} \sigma_{F^{\star}}(x_{0},x)^{n} dx  \bigg)^{\frac{1}{n}} \le \delta, $$
we have $D^{2}u \in  L^{\psi(\cdot)}(B_{1})$
with the following estimate
$$ ||D^2u||_{ L^{\psi(\cdot)}(B_{1})} \le C(||u||_{ L^{\infty}(B_{6})}^n+||f||_{ L^{\psi(\cdot)}(B_{6})})$$
for some $C=C(n,\lambda,\Lambda,p,q,L,t_{0})>0$.
\item[(b)]
Suppose that $f \in L^{\psi(\cdot)}(B_{6}^{+})$ and let $u$ be a viscosity solution to
\begin{align} \label{ob_mubdry}
\left\{ \begin{array}{ll}
F(D^{2}u,x) = f & \textrm{in $B_{6}^{+}, $}\\
u=0 & \textrm{on $T_{6}$.}
\end{array} \right.
\end{align}
Then there exists $\delta=\delta(n,\lambda,\Lambda,p,q,L,t_{0})>0$ such that
if  for any $r>0$ and $x_0\in \Omega$,
$$ \bigg( \kint_{B_{r}(x_{0})\cap B_1^+} \sigma_{F^{\star}}(x_{0},x)^{n} dx  \bigg)^{\frac{1}{n}} \le \delta, $$
we have $D^{2}u \in  L^{\psi(\cdot)}(B_{1}^{+})$
with the following estimate
$$ ||D^2 u||_{ L^{\psi(\cdot)}(B_{1}^{+})} \le C(||u||_{ L^{\infty}(B_{6}^{+})}^n+||f||_{ L^{\psi(\cdot)}(B_{6}^{+})})$$
for some $C=C(n,\lambda,\Lambda,p,q,L,t_{0})>0$.
\end{itemize}
\end{theorem}

The following result plays a key role in obtaining the above theorem.
\begin{lemma}\label{arr_ac}
Assume that $F=F(X,x)$ is uniformly elliptic with constants $\lambda$ and $\Lambda$ with $F(0,\cdot)=0$,
and there exists $F^{\star}=F^{\star}(X,x)$ which is convex in $X$. Then the following hold:
\begin{itemize}
\item[(a)] Suppose that $f \in L^{\psi(\cdot)}(B_{6})$ and let $u$ be a viscosity solution to
\begin{align} \label{ob_muin2}
F_{\mu}(D^{2}u,x) = f \quad \textrm{in} \ B_{6}
\end{align}
with
$$ ||u||_{L^{\infty}(B_{6})} \le 1 \quad \textrm{and} \quad
\kint_{B_{6}}\psi(|f|,x) dx \le 1.$$
Then there exist small $\mu,\delta>0$ depending on $n,\lambda,\Lambda,p,q,L$ and $t_{0}$ such that if
$$ \bigg( \kint_{B_{r}(x_{0})\cap B_1} \sigma_{F^{\star}}(x_{0},x)^{n} dx  \bigg)^{\frac{1}{n}} \le \delta, $$
we have
$$ \kint_{B_{1}} \psi\big(\Theta (u,B_1),x\big) dx \le C$$
for some $C=C(n,\lambda,\Lambda,p,q,L,t_{0})>0$.
\item[(b)]
Suppose that $f \in L^{\psi(\cdot)}(B_{6}^{+})$ and let $u$ be a viscosity solution to
\begin{align} \label{ob_mubdry2}
\left\{ \begin{array}{ll}
F_{\mu}(D^{2}u,x) = f & \textrm{in $B_{6}^{+}, $}\\
u=0 & \textrm{on $T_{6}$}
\end{array} \right.
\end{align}
with
$$ ||u||_{L^{\infty}(B_{6}^{+})} \le 1 \quad \textrm{and} \quad
\kint_{B_{6}^{+}}\psi(|f|,x) dx \le 1.$$
Then there exist small $\mu,\delta>0$ depending on $n,\lambda,\Lambda,p,q,L$ and $t_0$ such that if  for any $r>0$ and $x_0\in \Omega$,
$$\bigg( \kint_{B_{r}(x_{0})\cap B_1^+} \sigma_{F^{\star}}(x_{0},x)^{n} dx  \bigg)^{\frac{1}{n}} \le \delta, $$
 we have
$$ \kint_{B_{1}^{+}} \psi\big(\Theta(u,B_{1}^{+}),x\big) dx \le C$$
for some $C=C(n,\lambda,\Lambda,p,q,L,t_{0})>0$.
\end{itemize}
\end{lemma}

We only give the proof of (b), as (a) can also be shown in a similar way.
We first consider the case that $\psi$ is independent of $x$, that is, $\psi=\psi(t)$.
For the interior case, the corresponding result can be found in \cite[Theorem 2.4]{MR3957152}.
By using a similar argument in \cite{MR2486925}, one can also prove the following Hessian estimate in Orlicz spaces, which is an ingredient in obtaining our desired result.
\begin{lemma}\label{olbdest}
Let $\psi$ be a function depending only on $t$.
Then under the assumption of Lemma \ref{arr_ac},
there holds
$$ \kint_{B_{2}^{+}} \psi\big(\Theta(u,B_{2}^{+})\big) dx \le C$$
for some constant $C=C(n,\lambda,\Lambda,p,q,L,t_{0})>0$.
\end{lemma}

Now we prove Lemma \ref{arr_ac} in several steps. 
We begin the proof of the lemma with an approximation lemma (cf. \cite[Lemma 3.1]{MR3957152} and \cite[Proposition 3.2]{MR2486925}).

\begin{lemma} \label{asyconapp}
Let $u$ be a viscosity solution of
\begin{align*}
\left\{ \begin{array}{ll}
F_{\mu}(D^{2}u,x) = f & \textrm{in $B_{1}^{+}(0',\iota), $}\\
u=\varphi & \textrm{on $\partial B_{1}^{+}(0',\iota)$}
\end{array} \right.
\end{align*}
with $||u||_{L^{\infty}(B_{1}^{+}(0',\iota))} \le 1$, $f \in L^{n}(B_{1}^{+}(0',\iota)) $ and $0\le\iota\le1$,
where $\varphi\in C^{0,\alpha_0}(\partial B_{1}^{+}(0',\iota))$ satisfying $||\varphi||_{C^{0,\alpha_0}(\partial B_{1}^{+}(0',\iota))}\le C_0$ for some $0<\alpha_0\le1$ and $C_0\ge 1$.
Then for any $\epsilon >0$, there exists a small constant $\delta = \delta (\epsilon,n, \lambda, \Lambda,C_0)>0$  such that if
$$\mu + ||f||_{L^{n}( B_{1}^{+}(0',\iota)) }+||\sigma_{F^{\star}}(\cdot,0)||_{L^{n}(B_{1}^{+}(0',\iota)) } \le \delta,$$
then we have $$||u-v||_{L^{\infty}(B_{1}^{+}(0',\iota))} \le \epsilon ,$$
where $v$ is a viscosity solution of
\begin{align} \label{limpro1}
\left\{ \begin{array}{ll}
F^{\star}(D^{2}v,0) = 0 & \textrm{in $B_{1}^{+}(0',\iota), $}\\
v=\varphi  & \textrm{on $\partial B_{1}^{+}(0',\iota) $.} 
\end{array} \right.
\end{align}
\end{lemma}
\begin{proof}
Suppose not.
Then there exist $\{ F^{k} \}_{k=1}^{\infty}$, $\{ f_{k} \}_{k=1}^{\infty}$, $\{ \mu_{k} \}_{k=1}^{\infty}$, $\{ \nu_{k} \}_{k=1}^{\infty}$ and $\{ \varphi_{k} \}_{k=1}^{\infty}$
such that
\begin{align*}
\left\{ \begin{array}{ll}
F_{\mu_{k}}^{k}(D^{2}u_{k},x) = f_{k} & \textrm{in $B_{1}^{+}(0',\iota_k), $}\\
u_{k}=\varphi_{k} & \textrm{on $\partial B_{1}^{+}(0',\iota_k)$}
\end{array} \right.
\end{align*}
with $||u_{k}||_{L^{\infty}(B_{1}^{+}(0',\iota_k))} \le 1$,
$f_{k} \in L^{n}(B_{1}^{+}(0',\iota_k)) $, $||\varphi_k||_{C^{0,\alpha_0}(\partial B_{1}^{+}(0',\iota_k))}\le C_0$ and
$$\mu_{k} + ||f_{k}||_{L^{n}(B_{1}^{+}(0',\iota_k)) }+||\sigma_{(F^{k})^{\star}}(\cdot,0)||_{L^{n}(B_{1}^{+}(0',\iota_k)) } \le \frac{1}{k}$$
but \begin{align}\label{add1}||u_{k}-v_{k}||_{L^{\infty}(B_{1}^{+}(0',\iota_k))} \ge \epsilon_{0} ,
\end{align}
for some $\epsilon_{0}>0$ and  the viscosity solution $v_{k}$ to
\begin{align*}
\left\{ \begin{array}{ll}
(F^{k})^{\star}(D^{2}v_{k},0) = 0 & \textrm{in $B_{1}^{+}(0',\iota_k), $}\\
v_{k}=\varphi_{k}  & \textrm{on $\partial B_{1}^{+}(0',\iota_k) $.} 
\end{array} \right.
\end{align*}
We check by applying \cite[Theorem 1.10]{MR2486925} that $u_{k}$ satisfies
\begin{align*}
&||u_{k}||_{C^{0,\alpha_{1}}(B_{1}^{+}(0',\iota_k))}
\\& \le C(||u_{k}||_{L^{\infty}(B_{1}^{+}(0',\iota_k))}+||f_{k}||_{L^{n}(B_{1}^{+}(0',\iota_k))} +
||\varphi_{k}||_{C^{0,\alpha_0}(\partial B_{1}^{+}(0',\iota_k))} )
\\ & \le C \bigg( 1+ \frac{1}{k}   +C_0\bigg)
\\ & \le C(2+C_{0} )
\end{align*}
for some $0<\alpha_{1}<\alpha_{0}$ and $C>1$ depending on $n,\lambda$ and $\Lambda $.
Thus, we have
\begin{align*}
||u_{k}||_{C^{0,\alpha_{1}}(B_{1}^{+}(0',\iota_k))}  \le C(n,\lambda, \Lambda,C_{0}).
\end{align*}
By Arzel\`{a}-Ascoli criterion, there exist a monotone subsequence $\{ u_{k_{j}} \}$ and a function $u_{\infty}$
such that $u_{k_{j}}$ uniformly converges to $u_{\infty}$ as $j \to \infty$ in $\overline{B}_{1}^{+}(0',\iota_{\infty})$,
where $\iota_{\infty}=\lim_{j\to\infty}\iota_{k_j}$.

Fix $\eta >0$. Then according to \cite[Lemma 4.1]{MR3539473}, there exists a number $N$ depending on $\eta$ such that
$$ |F_{\mu_{l}}^{k}(X,x)-(F^{k})^{\star}(X,x)| \le \eta (1+ ||X||)$$
for every $l \ge N$, $X \in Sym(n)$ and $x \in B_{1}^{+}(0',\iota_{\infty})$.
In addition, since $(F^{k})^{\star}$ is uniformly elliptic,
$$ (F^{k})^{\star}(\cdot,0) \to G(\cdot,0) \quad \textrm{uniformly in every} \ K \subset \subset Sym(n)$$
for some uniformly elliptic functional $G(\cdot,0):Sym(n) \to \mathbb{R}$.
Now, for any $\varsigma \in C^{2}(\overline{B}_{1}^{+}(0',\iota_{\infty})) $,
we observe that
\begin{align} \label{addaf} \begin{split}
&|F_{\mu_{k}}^{k}(D^{2}\varsigma(x),x)-f_{k}(x)-G(D^{2}\varsigma(x),0)|
\\ & \le |F_{\mu_{k}}^{k}(D^{2}\varsigma(x),x)-F_{\mu_{k}}^{k}(D^{2}\varsigma(x),0)|+|f_{k}(x)|
\\& \ \hspace{-0.15em}+\hspace{-0.15em} | F_{\mu_{k}}^{k}(D^{2}\varsigma(x),0)\hspace{-0.15em}-\hspace{-0.15em}(F^{k})^{\star}(D^{2}\varsigma(x),0)|
\hspace{-0.15em}+\hspace{-0.15em}|(F^{k})^{\star}(D^{2}\varsigma(x),0)\hspace{-0.15em}- \hspace{-0.15em}G(D^{2}\varsigma(x),0)|
\\ & \le \bigg( \sup_{X \in Sym(n)} \frac{|F_{\mu_{k}}^{k}(X,x)-F_{\mu_{k}}^{k}(X,0)|}{||X||+1} \bigg)(||\varsigma||+1)+|f_{k}(x)|
\\& \ \hspace{-0.15em}+\hspace{-0.15em} | F_{\mu_{k}}^{k}(D^{2}\varsigma(x),0)\hspace{-0.15em}-\hspace{-0.15em}(F^{k})^{\star}(D^{2}\varsigma(x),0)|.
\end{split}
\end{align}
Similarly as in the proof of \cite[Lemma 3.1]{MR3957152}, we find that
\begin{align*}
&\int_{B_{1}^{+}(0',\iota_{\infty})} \bigg( \sup_{X \in Sym(n)}
\frac{|F_{\mu_{k}}^{k}(X,x)-F_{\mu_{k}}^{k}(X,0)|}{||X||+1} \bigg)^{n} dx \\ &
 \le C(n)\bigg\{\int_{B_{1}^{+}(0',\iota_{\infty})} \bigg( \sup_{X \in Sym(n)}
\frac{|F_{\mu_{k}}^{k}(X,x)-(F^{k})^{\star}(X,x)|}{||X||+1} \bigg)^{n} dx
 \\ & \quad + \sup_{X \in Sym(n)} \bigg( \frac{|F_{\mu_{k}}^{k}(X,0)-(F_{\mu_{k}})^{\star}(X,0)|}{||X||+1} \bigg)^{n} \hspace{-0.3em} + ||\sigma_{(F^{k})^{\star}}(x,0)||_{L^{n}(B_{1}^{+}(0',\iota_{\infty}))}^{n}\bigg\}.
\end{align*}
Then the right-hand side of \eqref{addaf} converges to $0$ in $L^{n}(B_{1}^{+}(0',\iota_{\infty}))$ as $k\to\infty$.
Combining \cite[Proposition 1.5]{MR2486925}, Arzel\`{a}-Ascoli criterion with our assumption, we obtain that $u_{\infty}$ and $v_{\infty}$ are solutions of
\begin{align*}
\left\{ \begin{array}{ll}
G(D^{2}u_{\infty},0) = 0 & \textrm{in $B_{1}^{+}(0',\iota_{\infty}), $}\\
u_{\infty}=\varphi_{\infty} & \textrm{on $\partial B_{1}^{+}(0',\iota_{\infty})$}.
\end{array} \right.
\end{align*}
By the uniqueness of the above problem, this is a contradiction to \eqref{add1} and we can complete the proof.
\end{proof}

The following $W^{2,\nu}$-type estimate is a building block to derive our desired result.
From now on, we write $\Theta = \Theta (u, B_{4}^{+})$ for simplicity.
\begin{lemma} \label{w2delta} Let $u$ be a function satisfying  $ u \in S(\lambda, \Lambda,  f) $ with $||u||_{L^{\infty}(B_{4}^{+})}\le 1$
and $\kint_{B_{6}^{+}}\psi(|f|,x) dx \le 1$.
Then for $\Theta = \Theta (u, B_{4}^{+})$,
we have $\Theta \in L^{\nu}(B_{4}^{+})$ with the following estimate
$$\bigg(\kint_{B_{4}^{+}}\Theta^{\nu} dx\bigg)^{1/\nu}\le C$$
for some $0<\nu<1$ and $C>0$ depending on $n,\lambda,\Lambda , L$ and $t_{0}$.
\end{lemma}
\begin{proof}
We first observe that
\begin{align*}
\bigg(\kint_{B_{4}^{+}}|f|^{n} dx\bigg)^{1/n}\le C
(\psi^{-})^{-1}\bigg( \kint_{B_{4}^{+}}\psi^{-}(|f|)dx\bigg)
\end{align*} for some $C=C(n,L)$ by using Lemma \ref{orljen}
since $\psi^{-}$ satisfies $\textrm{(aInc)}_{n}$. We can also see that
\begin{align*}
\kint_{B_{6}^{+}}\psi^{-}(|f|)dx \le \kint_{B_{6}^{+}}\psi(|f|,x)dx
\le 1.
\end{align*}
Therefore, we have
\begin{align*}
\bigg(\kint_{B_{6}^{+}}|f|^{n} dx\bigg)^{1/n} \le C(n,L)\big[(\varphi^{-})^{-1}(1)\big]^{1/n}\le C(n,L,t_0)
\end{align*}
by using the condition (A0). Applying Lemma \ref{w2del} (b), we have
\begin{align*} \bigg(\kint_{B_{4}^{+}}\Theta^{\nu} dx\bigg)^{1/\nu}&\le
C \bigg(||u||_{L^{\infty}(B_{6}^{+})} + \bigg(\kint_{B_{6}^{+}}|f|^{n} dx\bigg)^{1/n} \bigg)
\\ & \le C(n, \lambda, \Lambda, L, t_{0}).
\end{align*}
This completes the proof.
\end{proof}

Now we return to the proof of the main theorem.
We set
\begin{align} \label{lam0def} \lambda_{0} = \kint_{B_{2}^{+}}\psi(\Theta, x)^{\frac{\nu}{nq}} dx+
\frac{1}{\delta} \bigg(1+\kint_{B_{2}^{+}}\psi(|f|, x)^{\frac{1}{p}} dx\bigg)^{\frac{\nu p}{nq}}\ge 1,
\end{align}
where $\delta>0$ is to be determined.
Then, we can observe that
\begin{align*}
\kint_{B_{2}^{+}}\psi(\Theta, x)^{\frac{\nu}{nq}} dx \le \kint_{B_{2}^{+}}\psi_{B_{2}^{+}}^{+}(\Theta)^{\frac{\nu}{nq}} dx \le C \psi_{B_{2}^{+}}^{+}\bigg( \bigg(\kint_{B_{2}^{+}}\Theta^{\nu} dx\bigg)^{\frac{1}{nq}} \bigg) \le C_{0},    
\end{align*} 
where $C_0$ depends on $n,\lambda,\Lambda,L,q$ and $t_{0}$ by \cite[Lemma 3.1]{MR4288665} 
since $\psi$ satisfies $\textrm{(aDec)}_{nq}$.
By H\"{o}lder inequality, we have
$$ \bigg(\kint_{B_{2}^{+}}\psi(|f|, x)^{\frac{1}{p}} dx\bigg)^{\frac{\nu p}{nq}} \le \bigg(\kint_{B_{2}^{+}}\psi(|f|, x)dx\bigg)^{\frac{\nu}{nq}} \le 1.$$
This implies $\lambda_{0} \le C_{0}/\delta$.

We also define
\begin{align*}&E(s,\lambda):=\{ x\in B_{s}^{+}: \psi(\Theta,x)^{\frac{\nu}{nq}} > \lambda \},\\&
H(s,\lambda):=\{ x\in B_{s}^{+}: \psi(|f|,x)^{\frac{\nu}{nq}} > \lambda \}.
\end{align*}
Fix $s_{1}$ and $s_{2}$ with $1\le s_{1} < s_{2} \le 2$.

Then we can find a covering of upper level sets by using a stopping time argument (see also \cite{MR2352517}). 
\begin{lemma} \label{stta}
Let $\lambda \ge \alpha \lambda_{0}$ with $\alpha = \big(\frac{120}{s_{2}-s_{1}} \big)^{n}$.
Then there exists a disjoint family $\{ B_{\tau_{k}}^{+}(y^{k}) \}_{k=1}^{\infty}$ with $y^{k} \in E(s_{1},\lambda)$ and $\tau_{k} \in (0, \frac{s_{2}-s_{1}}{60})$
such that
$$ E(s_{1},\lambda) \subset \bigcup_{k=1}^{\infty}B_{5\tau_{k}}^{+}(y^{k}),$$
and $\mathcal{H}_{y^{k}}(\tau_{k})=\lambda $ and $\mathcal{H}_{y^{k}}(\tau)<\lambda $ for any $\tau \in (\tau_{k},s_{2}-s_{1}]$,
where
$$\mathcal{H}_{y}(\tau) := \kint_{B_{\tau}^{+}(y)}\psi(\Theta, x)^{\frac{\nu}{nq}} dx+
\frac{1}{\delta} \bigg(\kint_{B_{\tau}^{+}(y)}\psi(|f|, x)^{\frac{1}{p}} dx\bigg)^{\frac{\nu p}{nq}} .$$
\end{lemma}
\begin{proof}
We first observe that $ B_{\tau}^{+}(y)\subset B_{2}^{+}$ for any $0<\tau \le s_{2}-s_{1}$.
Therefore, for any $\frac{s_{2}-s_{1}}{60}\le\tau\le s_{2}-s_{1} $, we have
\begin{align*}
\mathcal{H}_{y}(\tau)& \le \bigg( \frac{2}{\tau} \bigg)^{n} \bigg\{
 \kint_{B_{2}^{+}}\psi(\Theta, x)^{\frac{\nu}{nq}} dx+
\frac{1}{\delta} \bigg(\kint_{B_{2}^{+}} \psi(|f|, x)^{\frac{1}{p}} dx\bigg)^{\frac{\nu p}{nq}}  \bigg\}
\\ & \le \bigg( \frac{120}{s_{2}-s_{1}} \bigg)^{n} \bigg\{
 \kint_{B_{2}^{+}}\psi(\Theta, x)^{\frac{\nu}{nq}} dx+
\frac{1}{\delta} \bigg(\kint_{B_{2}^{+}}\psi(|f|, x)^{\frac{1}{p}} dx\bigg)^{\frac{\nu p}{nq}}  \bigg\}
\\ & \le \alpha \lambda_{0}
\le \lambda.
\end{align*}

By Lebesgue's theorem, we deduce that
$$\lim_{\tau\to 0^+} \mathcal{H}_{y}(\tau) \ge \lim_{\tau\to 0^+} \kint_{B_{\tau}^{+}(y)} \psi(\Theta, x)^{\frac{\nu}{nq}} dx> \lambda  $$
for every $y \in E(s_{1},\lambda) $.
Thus, for any $y \in E(s_{1},\lambda)$, we can find
$ \tau_{y} =\tau(y) \in (0,\frac{ s_{2}-s_{1}}{60})$ such that
$$\mathcal{H}_{y}(\tau_{y})=\lambda \quad \textrm{and} \quad \mathcal{H}_{y}(\tau)<\lambda\quad \textrm{for every}
\ \tau \in (\tau_{y},  s_{2}-s_{1}]. $$
Now we can choose such $y^{k}$ and $\tau_{k}$ by using Vitali's covering lemma.
\end{proof}

From the above lemma, we can obtain the following estimate.
\begin{lemma} \label{ehest}
Under assumptions in Lemma \ref{stta}, we have
\begin{align*}
|B_{\tau_{k}}^{+}(y^{k})|\le \frac{4}{\lambda} \int_{E_{k}(\lambda/4)}\psi(\Theta, x)^{\frac{\nu}{nq}} dx +
 \bigg( \frac{4}{\lambda \delta}\bigg)^{\frac{nq}{\nu p}}\int_{H_{k}(\lambda \delta/4)}\psi(|f|, x)^{\frac{1}{p}} dx ,
\end{align*}
where for $\lambda>0$,
\begin{align*}&E_{k}(\lambda)=\{ x\in B_{\tau_k}^{+}(y^{k}): \psi(\Theta,x)^{\frac{\nu}{nq}} \ge \lambda \},\\&
H_{k}(\lambda)=\{ x\in B_{\tau_k}^{+}(y^{k}): \psi(|f|,x)^{\frac{\nu}{nq}} \ge \lambda \}.
\end{align*}
\end{lemma}
\begin{proof}
By Lemma \ref{stta}, there exist $y^{k}$ and $\tau_{k}$ such that $\mathcal{H}_{y^{k}}(\tau_{k})=\lambda $ and $\mathcal{H}_{y^{k}}(\tau)<\lambda $ for any $\tau \in (\tau_{k}, s_{2}-s_{1}]$. Then
$$ \lambda = \kint_{B_{\tau_{k}}^{+}(y^{k})}\psi(\Theta, x)^{\frac{\nu}{nq}} dx+
\frac{1}{\delta} \bigg(\kint_{B_{\tau_{k}}^{+}(y^{k})}\psi(|f|, x)^{\frac{1}{p}} dx\bigg)^{\frac{\nu p}{nq}}.$$
Then we can observe that
\begin{align*} \frac{\lambda}{2} \le  \kint_{B_{\tau_{k}}^{+}(y^{k})}\psi(\Theta, x)^{\frac{\nu}{nq}} dx  \quad \textrm{or} \quad \frac{\lambda\delta}{2}
\le \bigg( \kint_{B_{\tau_{k}}^{+}(y^{k})}\psi(|f|, x)^{\frac{1}{p}} dx \bigg)^{\frac{\nu p}{nq}}.
\end{align*}

Now we have
\begin{align*}
\frac{\lambda}{2}|B_{\tau_{k}}^{+}(y^{k})| & \le  \int_{B_{\tau_{k}}^{+}(y^{k})}\psi(\Theta, x)^{\frac{\nu}{nq}} dx
\\ & \le \int_{E_{k}(\lambda/4)}\psi(\Theta, x)^{\frac{\nu}{nq}} dx + \frac{\lambda}{4}|B_{\tau_{k}}^{+}(y^{k})|
\end{align*}
and this implies
$$ |B_{\tau_{k}}^{+}(y^{k})|  \le  \frac{4}{\lambda} \int_{E_{k}(\lambda/4)}\psi(\Theta, x)^{\frac{\nu}{nq}} dx . $$

Similarly, we can also obtain
\begin{align*}
\bigg(\frac{\lambda \delta}{2} \bigg)^{\frac{nq}{\nu p}}|B_{\tau_{k}}^{+}(y^{k})|  \le
\int_{H_{k}(\lambda \delta/4)}\psi(|f|, x)^{\frac{1}{p}} dx +
\bigg(\frac{\lambda \delta}{4} \bigg)^{\frac{nq}{\nu p}}|B_{\tau_{k}}^{+}(y^{k})|
\end{align*}
and this estimate yields
$$|B_{\tau_{k}}^{+}(y^{k})|  \le \bigg( \frac{4}{\lambda\delta}\bigg)^{\frac{nq}{\nu p}}\int_{H_{k}(\lambda\delta/4)}\psi(|f|, x)^{\frac{1}{p}} dx  $$
since $\lambda\ge \alpha\lambda_0>1$ and $\nu p/nq <1$.
Now we can complete the proof.
\end{proof}

Now we define
$$\psi_{k}^{+}(t)=\sup_{x\in B_{5\tau_{k}}^{+}(y^{k})}\psi(t,x) \quad
\textrm{and} \quad \psi_{k}^{-}(t)=\inf_{x\in B_{5\tau_{k}}^{+}(y^{k})}\psi(t,x).$$
Then we can also observe the following result.
\begin{lemma} \label{lemnon}
Under assumptions in Lemma \ref{stta}, the following hold:
\begin{itemize}
\item[(a)] If $B_{30\tau_{k}}(y^{k}) \subset B_{s_{2}}^{+}$, we have
\begin{align*}
\bigg(\kint_{B_{30\tau_{k}}(y^{k})}\Theta^{\frac{\nu p}{q}} dx \bigg)^{^{\frac{q}{\nu p}}} \le C (\psi_{k}^{+})^{-1}\big(\lambda^{\frac{nq}{\nu}} \big)
\end{align*}and
\begin{align*}
\bigg(\kint_{B_{30\tau_{k}}(y^{k})}|f|^{n} dx \bigg)^{^{\frac{1}{n}}} \le C\delta^{\frac{q}{\nu p}} (\psi_{k}^{+})^{-1}\big(\lambda^{\frac{nq}{\nu}} \big)
\end{align*}
for some $C=C(n,\lambda,\Lambda,p,q,L,t_{0})>0$.
\item[(b)]  If $B_{30\tau_{k}}(y^{k}) \not\subset B_{s_{2}}^{+}$, we have
\begin{align*}
\bigg(\kint_{B_{60\tau_{k}}^{+}(\tilde{y}^{k}) }\Theta^{\frac{\nu p}{q}} dx \bigg)^{^{\frac{q}{\nu p}}} \le C (\psi_{k}^{+})^{-1}\big(\lambda^{\frac{nq}{\nu}} \big)
\end{align*}and
\begin{align*}
\bigg(\kint_{B_{60\tau_{k}}^{+}(\tilde{y}^{k}) }|f|^{n} dx \bigg)^{^{\frac{1}{n}}} \le C\delta^{\frac{q}{\nu p}} (\psi_{k}^{+})^{-1}\big(\lambda^{\frac{nq}{\nu}} \big)
\end{align*}
for some $C=C(n,\lambda,\Lambda,p,q ,L,t_{0})>0$, where $\tilde{y}^{k}=\big( (y^{k})',0 \big)$.
\end{itemize}
\end{lemma}
\begin{proof} 
We first consider the case $B_{30\tau_{k}} (y^{k})\subset B_{s_{2}}^{+}$.
Since $H_{y}^k(\tau) < \lambda$ for any $\tau \in [\tau_{k},  s_{2}-s_{1}]$ from Lemma \ref{stta},
we have
$$\kint_{B_{30\tau_{k}}(y^{k})}\psi(\Theta, x)^{\frac{\nu}{nq}} dx  \le \lambda
\quad \textrm{and} \quad \bigg(\kint_{B_{30\tau_{k}}(y^{k})}\psi(|f|, x)^{\frac{1}{p}} dx\bigg)^{\frac{\nu p}{nq}} \le \lambda\delta . $$

Setting $\theta (t)=\psi_{k}^{-}\big(t^{\frac{q}{\nu p}}\big)^{\frac{\nu}{nq}}$,
we see that $\theta$ satisfies $\textrm{(aInc)}_{1}$. Hence, by Lemma \ref{orljen},
we have
\begin{align*}
\psi_{k}^{-}\bigg( C \bigg( \kint_{B_{30\tau_{k}}(y^{k})} \Theta^{\frac{\nu p}{q}}dx
\bigg)^{\frac{q}{\nu p}} \bigg) &\le
\bigg( \kint_{B_{30\tau_{k}}(y^{k})} \psi_{k}^{-}(\Theta)^{\frac{\nu}{nq}} dx \bigg)^{\frac{nq}{\nu}}
\\ & \le \bigg( \kint_{B_{30\tau_{k}}(y^{k})} \psi(\Theta,x)^{\frac{\nu}{nq}} dx \bigg)^{\frac{nq}{\nu}} \\ & \le \lambda^{\frac{nq}{\nu}},
\end{align*}
where $C$ depends on $n$.

We also observe that
\begin{align*}
 \bigg( \kint_{B_{30\tau_{k}}(y^{k})}  \Theta^{\frac{\nu p}{q}}dx \bigg)^{\frac{nq}{\nu p}}
& \le \kint_{B_{30\tau_{k}}(y^{k})} \Theta^n dx
\\& \le C(\varphi^{-})^{-1}\bigg( \kint_{B_{30\tau_{k}}(y^{k})}
 \varphi^{-}( \Theta^n) dx \bigg) \\ & \le C_0(\varphi^{-})^{-1} \bigg(\frac{1}{|B_{5\tau_{k}}|} \bigg)
\end{align*} for some $C_0=C_0(n,\lambda,\Lambda,p,q,L,t_{0})>0$
from Lemma \ref{olbdest}.
Thus,
\begin{align*}
\varphi_{k}^{+} \bigg(\bigg( \kint_{B_{30\tau_{k}}(y^{k})} \Theta^{\frac{\nu p}{q}}dx \bigg)^{\frac{nq}{\nu p}} \bigg)&\le  C\bigg( \varphi_{k}^{+} \bigg( \frac{1}{C_0}\bigg( \kint_{B_{30\tau_{k}}(y^{k})} \Theta^{\frac{\nu p}{q}}dx \bigg)^{\frac{nq}{\nu p}}\bigg)  +1\bigg)
\\ & \le   C\bigg( \varphi_{k}^{-} \bigg(\bigg( \kint_{B_{30\tau_{k}}(y^{k})} \Theta^{\frac{\nu p}{q}}dx \bigg)^{\frac{nq}{\nu p}}\bigg)  +1\bigg)
\\ & = C\bigg( \psi_{k}^{-} \bigg(\bigg( \kint_{B_{30\tau_{k}}(y^{k})} \Theta^{\frac{\nu p}{q}}dx \bigg)^{\frac{q}{\nu p}}\bigg)  +1\bigg)
\\ & \le C\lambda^{\frac{nq}{\nu}},
\end{align*}
where $C=C(n,\lambda,\Lambda,p,q,L,t_{0})>0$, which implies
$$\bigg( \kint_{B_{30\tau_{k}}(y^{k})} \Theta^{\frac{\nu p}{q}}dx \bigg)^{\frac{q}{\nu p}}\le
C (\psi_k^{+})^{-1}(\lambda^{\frac{nq}{\nu}}).$$
Note that we have used the assumption (A1-$\varphi^{-}$) to derive the second inequality.

Repeating a similar calculation, we finally obtain
\begin{align*}
\psi_{k}^{+}\bigg(\bigg(\kint_{B_{30\tau_{k}}(y^{k})}|f|^{n} dx \bigg)^{^{\frac{1}{n}}} \bigg) \le
C\big( \lambda\delta \big)^{\frac{nq}{\nu}}
\end{align*}
and this yields
$$\bigg(\kint_{B_{30\tau_{k}}(y^{k})}|f|^{n} dx \bigg)^{^{\frac{1}{n}}} \le C \delta^{\frac{q}{\nu p}}
(\psi_{k}^{+})^{-1}\big( \lambda^{\frac{nq}{\nu}} \big),$$
where $C=C(n,\lambda,\Lambda,p,q,L,t_{0})>0$. 

Next we consider the case $B_{30\tau_{k}}(y^{k}) \not\subset B_{s_{2}}^{+}$.
Observe that $|y^{k}-\tilde{y}^{k}|=|(y^{k})_{n}| < 30\tau_{k}$, $60\tau_{k} \le  s_{2}-s_{1} \le 1$ and
$y^{k} \in B_{s_{1}}^{+}$. Then we see that
$$B_{30\tau_{k}}^{+}(y^{k})\subset B_{60\tau_{k}}^{+}(\tilde{y}^{k}) \quad \textrm{and} \quad      B_{60\tau_{k}}^{+}(\tilde{y}^{k})\subset B_{s_1+60\tau_{k}}^{+} \subset  B_{s_{2}}^{+}.$$
Thus, by using a similar argument in the first case, we obtain
\begin{align*}
\bigg(\kint_{B_{60\tau_{k}}^{+}(\tilde{y}^{k}) }\Theta^{\frac{\nu p}{q}} dx \bigg)^{^{\frac{q}{\nu p}}} \le C (\psi_{k}^{+})^{-1}\big(\lambda^{\frac{nq}{\nu}} \big)
\end{align*}and
\begin{align*}
\bigg(\kint_{B_{60\tau_{k}}^{+}(\tilde{y}^{k}) }|f|^{n} dx \bigg)^{^{\frac{1}{n}}} \le C\delta^{\frac{q}{\nu p}} (\psi_{k}^{+})^{-1}\big(\lambda^{\frac{nq}{\nu}} \big)
\end{align*}
for some $C=C(n,\lambda,\Lambda,p,q,L,t_{0})>0$ by using the result of Lemma \ref{olbdest}.
\end{proof}

From the above results, we can deduce the following lemma.
This lemma can be regarded as Hessian estimates for a solution to the limiting problem of 
\begin{align}\label{acsc}
\left\{ \begin{array}{ll}
F_{\mu}(D^{2}u,x) = f & \textrm{in $B_{6}^{+}, $}\\
u=0 & \textrm{on $T_{6}$,} 
\end{array} \right.
\end{align}
and the difference between them and our original solutions. 
\begin{lemma} \label{thetaest}Under the assumption as above, the following hold:
\begin{itemize}
\item[(a)] For any $0<\epsilon<1$, there exist
$$\delta=\delta(n, \lambda, \Lambda, p,q,L,t_{0},\epsilon)\in(0,1)$$
given in 
\eqref{lam0def} and $h_{k} \in C^{2}(B_{2})$ such that
$$||\Theta(h_{k},B_{2}) ||_{L^{\infty}(B_{5\tau_{k}}(y^{k}))}
\le C_2(\psi_{k}^{+})^{-1}\big( \lambda^{\frac{nq}{\nu}} \big)$$
for some $C_2=C_2(n,\lambda,\Lambda,p,q,L,t_{0})\ge1$ and
$$ ||\Theta(u-h_{k},B_{2}) ||_{L^{\nu}(B_{5\tau_{k}}(y^{k}))}
\le \epsilon(\psi_{k}^{+})^{-1}\big( \lambda^{\frac{nq}{\nu}} \big)|B_{5\tau_{k}}(y^{k})|^{\frac{1}{\nu}}.$$
\item[(b)] For any $0<\epsilon<1$, there exist
$$\delta=\delta(n, \lambda, \Lambda, p,q,L,t_{0},\epsilon)\in(0,1)$$
given in 
\eqref{lam0def} and $h_{k} \in C^{2}(B_{2}^{+})$ such that
$$||\Theta(h_{k},B_{2}^{+}) ||_{L^{\infty}(B_{5\tau_{k}}^{+}(y^{k}))}
\le C_2(\psi_{k}^{+})^{-1}\big( \lambda^{\frac{nq}{\nu}} \big)$$
for some $C_2=C_2(n,\lambda,\Lambda,p,q,L,t_{0})\ge 1$ and
$$ ||\Theta(u-h_{k},B_{2}^{+}) ||_{L^{\nu}(B_{5\tau_{k}}^{+}(y^{k}))}
\le \epsilon(\psi_{k}^{+})^{-1}\big( \lambda^{\frac{nq}{\nu}} \big)|B_{5\tau_{k}}^{+}(y^{k})|^{\frac{1}{\nu}}.$$
\end{itemize}
\end{lemma} 
\begin{proof} Without loss of generality, we only consider the boundary case (b), i.e. $B_{5\tau_{k}}(y^{k}) \not\subset B_{s_{2}}^{+} $. We recall the notation $\tilde{y}^{k}$ in the proof of Lemma \ref{lemnon}.

Set
$$\tilde{u}(y)=\frac{u(10\tau_{k}y+\tilde{y}^{k})}{C_{0}(10\tau_{k})^{2} (\psi_{k}^{+})^{-1}\big(\lambda^{\frac{nq}{\nu}} \big) }, \
\tilde{f}(y) =\frac{f(10\tau_{k}y+\tilde{y}^{k})}{C_{0} (\psi_{k}^{+})^{-1}\big(\lambda^{\frac{nq}{\nu}} \big) },  $$
where $C_{0}$ is the constant $C$ in Lemma \ref{lemnon}.
Then, $\tilde{u}$ is a viscosity solution to
\begin{align*}
\left\{ \begin{array}{ll}
G(D^{2}\tilde{u},x) = \tilde{f} & \textrm{in $B_{6}^{+}, $}\\
\tilde{u}=0 & \textrm{on $T_{6},$}
\end{array} \right.
\end{align*}
where
$$G(Y,y)=\frac{1}{C_{0} (\psi_{k}^{+})^{-1}\big(\lambda^{\frac{nq}{\nu}}\big)}
F\big( C_{0} (\psi_{k}^{+})^{-1}\big(\lambda^{\frac{nq}{\nu}} \big) Y,
10\tau_{k}y+\tilde{y}^{k} \big) .$$
And we also define
$$ \tilde{\Theta}(y)=\Theta(\tilde{u}, \mathcal{T}^{-1}B_{2}^{+} )(y) \quad
\textrm{where} \quad \mathcal{T}y= 10\tau_{k}y+\tilde{y}^{k}.$$ 

Observe that
$||\tilde{f}||_{L^{n}(B_{6}^{+})} \le C \delta^{\frac{q}{\nu p}}$,
$||\sigma_{G^{\star}}(\cdot,0)||_{L^{n}(B_{6}^{+})} \le C\delta $ for some $C=C(n)$ and
$$ \bigg(  \kint_{B_{6}^{+}} \tilde{\Theta}^{\frac{\nu p}{q}}dx \bigg)^{\frac{q}{\nu p}} \le 1.$$
We can find a point $y_{0} \in B_{6}^{+}$ such that $\tilde{\Theta}(y_{0}) \le 1$, that is,
there exists a linear function $L$ such that  
$$ |(\tilde{u}-L)(y)| \le \frac{1}{2}|y-y_{0}|^{2} \quad \textrm{for any} \ y \in
\mathcal{T}^{-1}B_{2}^{+}. $$
Set $\tilde{u}_{1}=(\tilde{u}-L)/C_{1} $ for $C_{1}=C_{1}(n)>0$
sufficiently large such that
\begin{align} \label{esttu1} ||\tilde{u}_{1}||_{L^{\infty}(B_{6}^{+})} \le 1
\quad \textrm{and} \quad |\tilde{u}_{1}(y)| \le \frac{1}{2C_{1}}|y-y_{0}|^{2}
\ \ \textrm{in} \ \mathcal{T}^{-1}B_{2}^{+} .
\end{align} 

By Lemma \ref{asyconapp}, for a function $\tilde{v}$ solving
\begin{align*}
\left\{ \begin{array}{ll}
G^{\star}(D^{2}\tilde{v},0) = 0 & \textrm{in $B_{5}^{+}, $}\\
\tilde{v}=\tilde{u}_{1} & \textrm{on $\partial B_{5}^{+}  $} 
\end{array} \right.
\end{align*}
in the viscosity sense, we can choose some $\delta$ depending on $n, \lambda, \Lambda, p,q$ and $\eta$ such that $||\tilde{u}_{1}-\tilde{v}||_{L^{\infty}(B_{4}^{+})} \le \eta$ for any $\eta>0$.
We also choose  $\mu<C \delta^{\frac{q}{\nu p}}$.
Since $G^{\star}$ is convex, we have 
\begin{align} \label{c2estob} \begin{split}
||\Theta(\tilde{v},B_{4}^{+})||_{L^{\infty}(B_{4}^{+})}\le C(n,\lambda,\Lambda )
||\tilde{u}_{1}||_{L^{\infty}(\partial B_{5}^{+})}
\le C(n,\lambda,\Lambda)
\end{split}
\end{align}
by using a similar procedure to prove \cite[Theorem 4.3]{MR2486925}.

Since $\tilde{u}_{1}$ solves
\begin{align*}
\left\{ \begin{array}{ll}
\tilde{G}_{\mu}(D^{2}\tilde{u}_{1},x) = \frac{\tilde{f}}{C_{1}} & \textrm{in $B_{6}^{+}, $}\\
\tilde{u}_{1}=-\frac{L}{C_{1}}  & \textrm{on $T_{6},$}
\end{array} \right.
\end{align*}
where
$$\tilde{G}_{\mu}(Y,y)=\frac{1}{C_{1}}G_{\mu}(C_{1}Y,y), $$
we have
\begin{align*} 
\left\{ \begin{array}{ll}
\tilde{u}_{1}-\tilde{v} \in S(\frac{\lambda}{n}, \Lambda,  \frac{\tilde{f}}{C_{1}}-G^{\star}(D^{2}\tilde{v},\cdot)) & \textrm{in $B_{5}^{+}, $}\\
\tilde{u}_{1}-\tilde{v} \equiv 0& \textrm{on $\partial B_{5}^{+}  $.} 
\end{array} \right.
\end{align*}
Now we can see that
\begin{align}\label{gstarest}\begin{split}
|G^{\star}(D^{2}\tilde{v}(y),y)| &\le \sigma_{G^{\star}}(\cdot,0)  (||D^{2}\tilde{v}||_{L^{\infty}(B_{4}^{+})}+1) \\
& \le C\sigma_{G^{\star}}(\cdot,0) \bigg( ||\tilde{u}_{1}||_{L^{\infty}(B_{6}^{+})} +
\bigg|\bigg| \frac{\tilde{f}}{C_{1}}  \bigg|\bigg|_{L^{n}(B_{6}^{+})} +1 \bigg)\end{split}
\end{align}
for any $y \in B_{4}^{+}$ and some $C=C(n,\lambda,\Lambda)$
(see also the proof of \cite[Lemma 7.9]{MR1351007}).
Thus, we have
\begin{align*}
||G^{\star}&(D^{2}\tilde{v},\cdot)||_{L^{n}(B_{4}^{+})} \\ & \le
C||\sigma_{G^{\star}}(\cdot,0) ||_{L^{n}(B_{4}^{+})}\bigg( ||\tilde{u}_{1}||_{L^{\infty}(B_{6}^{+})} +
\bigg|\bigg| \frac{\tilde{f}}{C_{1}}  \bigg|\bigg|_{L^{n}(B_{6}^{+})} +1 \bigg)
\\ & \le C\delta(2 +\delta^{\frac{q}{\nu p}} )
\\ & \le C\delta,
\end{align*}
where $C= C(n,\lambda,\Lambda)$.
By Lemma \ref{w2delta}, we obtain that
\begin{align} \begin{split} \label{difldelest}
||\Theta& (\tilde{u}_{1}-\tilde{v}, B_{2}^{+})||_{L^{\nu}(B_{2}^{+})} \\& \le
C\bigg(||\tilde{u}_{1}-\tilde{v}||_{L^{\infty}(B_{2}^{+})}+
\bigg|\bigg|\frac{\tilde{f}}{C_{1}}-G^{\star}(D^{2}\tilde{v}(y),y)\bigg|\bigg|_{L^{n}(B_{2}^{+})}\bigg)\\ &
\le C(\eta + \delta^{\frac{q}{\nu p}} + \delta)
\\ & \le C(n,\lambda,\Lambda)\eta.\end{split}
\end{align}
We choose sufficiently small $\delta$ to obtain the last inequality.

Now we extend $\tilde{v}$ to $ \mathcal{T}^{-1}B_{2}^{+}$ satisfying properties similar to \eqref{c2estob} and \eqref{difldelest}.
Define $\tilde{h}\in C(\mathcal{T}^{-1}B_{2}^{+})$ as
\begin{align*} \tilde{h}(y)=
\left\{ \begin{array}{ll}
\tilde{v}(y) & \textrm{in $B_{2}^{+}, $}\\
\tilde{v}(y)+ \frac{|y|-2}{2}(\tilde{u}_{1}-\tilde{v})(y)  & \textrm{in $B_{4}^{+} \backslash B_2^{+}$,} \\
\tilde{u}_{1}(y) & \textrm{in $\mathcal{T}^{-1}B_{2}^{+} \backslash B_{4}^{+} $}.
\end{array} \right.
\end{align*}
Then we have
\begin{align*} (\tilde{u}_{1}-\tilde{h})(y)=
\left\{ \begin{array}{ll}
(\tilde{u}_{1}-\tilde{v})(y) & \textrm{in $B_{2}^{+}, $}\\
\frac{4-|y|}{2}(\tilde{u}_{1}-\tilde{v})  & \textrm{in $B_{4}^{+} \backslash B_{2}^{+}$,} \\
0 & \textrm{in $\mathcal{T}^{-1}B_{2}^{+} \backslash B_{4}^{+} $}.
\end{array} \right.
\end{align*} 

Suppose that $z_{0} \in B_{1}^{+}$ and $m>\Theta(\tilde{v},B_{2}^{+})(z_{0})$.
Then, we have
\begin{align} \label{estthm}
|\tilde{h}(y)-\tilde{L}(y)| =| \tilde{v}(y)- \tilde{L}(y)| \le  \frac{m}{2}|y-z_{0}|^{2} ,
\end{align}
where $\tilde{L}(y)= \tilde{v}(z_{0})+A \cdot (y-z_{0})$
for any $y \in B_{2}^{+}$ and some $A \in \mathbb{R}^{n}$.
We also observe that $||\tilde{v}||_{L^{\infty}(B_{4}^{+})} \le 2$ by using \eqref{esttu1} and $||\tilde{u}_{1}-\tilde{v}||_{L^{\infty}(B_{4}^{+})} \le \eta$ for any $\eta <1$.
Note that we can choose $y \in \partial B_{1}(z_{0}) \cap B_{2}^{+}$ such that
$A \cdot (y-z_{0})= -|A|$ when $A_{n} \le 0$ and
$A \cdot (y-z_{0})= |A| $ when $A_{n} > 0$.
Then from \eqref{estthm}, we have
\begin{align} \label{graest}
|A|\le \frac{m}{2}+4.
\end{align}
Now we can observe that
\begin{align*}
\tilde{h}(y)&\le ||\tilde{v}||_{L^{\infty}(B_{4}^{+})}+|\tilde{u}_1(y)|
\\&\le 2+\frac{|y-y_{0}|^{2}}{2C_{1}}
\\ & \le 2|y-z_{0}|^{2}+\frac{1}{C_{1}}|y-z_{0}|^{2}+\frac{1}{C_{1}}|z_{0}-y_{0}|^{2}
\\ & \le 2|y-z_{0}|^{2}+\frac{1}{C_{1}}|y-z_{0}|^{2} + \frac{49}{C_{1}}
|y-z_{0}|^{2}
\end{align*}
for any $y \in \mathcal{T}^{-1}B_{2}^{+}  \backslash B_{2}^{+} $,
since $y_{0} \in B_{6}^{+}$.
Using \eqref{graest}, we obtain
\begin{align*}
\tilde{h}(y)\le
\bigg\{& \frac{2C_{1}+50}{C_{1}} +\bigg(
 \frac{m}{2}+4\bigg)+2 \bigg\}|y-z_{0}|^{2}
  + A\cdot(y-z_{0})+\tilde{v}(z_{0})
\end{align*}
for any $y \in \mathcal{T}^{-1}B_{2}^{+}  \backslash B_{2}^{+} $.
Similarly, we can also derive that
\begin{align*}
\tilde{h}(y)\ge -
\bigg\{& \frac{2C_{1}+50}{C_{1}} +\bigg(
 \frac{m}{2}+4\bigg)+2 \bigg\}|y-z_{0}|^{2}
 + A\cdot(y-z_{0})+\tilde{v}(z_{0})
\end{align*}
for any $y \in \mathcal{T}^{-1}B_{2}^{+}  \backslash B_{2}^{+} $.
Now we see that
\begin{align*}
\Theta(\tilde{h},\mathcal{T}^{-1}B_{2}^{+})(z_{0}) \le
\Theta(\tilde{v},B_{2}^{+})(z_{0})+  \frac{2(8C_{1}+50)}{C_{1}} \le C
\end{align*}
because $m$ is arbitrary and $||\Theta(\tilde{v},B_{2}^{+})||_{L^{\infty}(B_{2}^{+})}\le C$ for some constant
$C=C(n,\lambda,\Lambda )$ and
this shows $$||\Theta(\tilde{h},\mathcal{T}^{-1}B_{2}^{+})(z_{0}) || \le C. $$

We next suppose that $z_{1} \in B_{1}^{+}$ and $\tilde{m} >\Theta(\tilde{u}_{1}-\tilde{v}
,B_{2}^{+})(z_{1})$.
Then, for any $y \in B_{2}^{+}$, there is a linear function $\overline{L}$ such that
\begin{align*}
|(\tilde{u}_{1}-\tilde{h})(y)-\overline{L}(y)|=|(\tilde{u}_{1}-\tilde{v})(y)-\overline{L}(y)|
\le \frac{\tilde{m}}{2}|y-z_{1}|^{2}.
\end{align*}
We can write $\overline{L}(y)=\tilde{A}\cdot(y-z_{1})+(\tilde{u}_{1}-\tilde{v})(z_{1})$ for some
$\tilde{A} \in \mathbb{R}^{n}$.
Recall $||\tilde{u}_{1}-\tilde{v}||_{L^{\infty}(B_{4}^{+})} \le \eta$ for any $\eta>0$.
Then by using a similar argument as above, we can see that
\begin{align} \label{tgraest}
|\tilde{A}| \le \frac{\tilde{m}}{2}+2\eta.
\end{align}

Observe that
\begin{align*}
0 \le \bigg( \frac{\tilde{m}}{2}+3\eta\bigg)|y-z_{1}|^{2}
+\tilde{A}\cdot(y-z_{1})+ (\tilde{u}_{1}-\tilde{h})(z_{1})
\end{align*}
for any $y \in  \mathcal{T}^{-1}B_{2}^{+}  \backslash B_{2}^{+}  $.
We now see that
\begin{align*}
(\tilde{u}_{1}-\tilde{h})(y) & \le ||\tilde{u}_{1}-\tilde{v}||_{L^{\infty}(B_{4}^{+})}
\\& \le \eta|y-z_{1}|^{2}
\\ & \le \bigg( \frac{\tilde{m}}{2}+ 4\eta \bigg)|y-z_{1}|^{2}
+\tilde{A}\cdot(y-z_{1})+ (\tilde{u}_{1}-\tilde{h})(z_{1})
\end{align*}
for any $y \in  \mathcal{T}^{-1}B_{2}^{+}  \backslash B_{2}^{+}  $.
Note that we can show such an inequality for the reverse direction.
Hence, we obtain
\begin{align*}
\Theta(\tilde{u}_{1}-\tilde{h}, \mathcal{T}^{-1}B_{2}^{+} )(z_{1}) \le
\Theta(\tilde{u}_{1}-\tilde{v},B_{2}^{+})(z_{1})  +8\eta.
\end{align*}
Now we deduce that
$$ ||\Theta(\tilde{u}_{1}-\tilde{h},\mathcal{T}^{-1}B_{2}^{+} )||_{L^{\nu}(B_{1}^{+})} \le C\eta $$
from \eqref{difldelest}, where $C = C(n,\lambda,\Lambda)$.
Since $B_{5\tau_{k}}^{+}(y^{k})\subset B_{10\tau_{k}}^{+}(\tilde{y}^{k}) $,
we can complete the proof by scaling back and choosing $\eta$ sufficiently small.
\end{proof}

Now we estimate $|E(s_{1},K\lambda)|$ for any $\lambda \ge \alpha \lambda_{0} \ge 1$.
Set $K_1: =4C_{2}^{\frac{\nu p}{q}}$ where $C_{2}$ is the constant in Lemma \ref{thetaest}.
Since $K_1 \ge 1$, we see that $E(s_{1},K_1\lambda) \subset E(s_{1},\lambda)$.
Then, by Lemma \ref{stta}, there exist
$\{ B_{\tau_{k}}^{+}(y^{k}) \}_{k=1}^{\infty}$ with $y^{k} \in E(s_{1},\lambda)$ and $\tau_{k}\in (0, \frac{ s_{2}-s_{1}}{60})$
such that
$$ E(s_{1},K_1\lambda) \subset  E(s_{1},\lambda) \subset \bigcup_{k=1}^{\infty}B_{5\tau_{k}}^{+}(y^{k}).  $$

Observe that
$$\Theta(u,B_{2}^{+})\le \Theta(u-h_{k},B_{2}^{+})+\Theta(h_{k},B_{2}^{+}). $$
For simplicity, we omit the domain for $\Theta$ and write $\Theta(u,B_{2}^{+}) =\Theta(u) $.
Let $K_2=C_{2}^{\nu}2^{(nq-1)\frac{\nu}{nq}+1}$ and we define $K=\max\{K_1,K_2\}$.
Then for $x \in B_{5\tau_{k}}^{+}(y^{k}) \cap  E(s_{1},K\lambda)$, we have
\begin{align}\label{labad1}\begin{split}
\psi(\Theta, x)^{\frac{\nu}{nq}}&\le \psi(\Theta(u-h_{k})+\Theta(h_{k}), x)^{\frac{\nu}{nq}}
\\ & \le \bigg( \frac{1}{2} \psi_{k}^{+}(2\Theta(u-h_{k}))+ \frac{1}{2} \psi_{k}^{+}(2\Theta(h_{k}))\bigg)^{\frac{\nu}{nq}}
\\& \le \bigg( \frac{1}{2}\cdot 2^{nq} \psi_{k}^{+}(\Theta(u-h_{k}))+ \frac{1}{2} \psi_{k}^{+}(2C_{2}\big(\psi_{k}^{+})^{-1}\big( \lambda^{\frac{nq}{\nu}})\big)\bigg)^{\frac{\nu}{nq}}
\\ & \le \bigg( 2^{nq-1} \psi_{k}^{+}(\Theta(u-h_{k}))+ C_{2}^{nq}2^{nq-1} \lambda^{\frac{nq}{\nu}}\bigg)^{\frac{\nu}{nq}}
\\ & \le  2^{(nq-1)\frac{\nu}{nq}} \psi_{k}^{+}(\Theta(u-h_{k}))^{\frac{\nu}{nq}}+\frac{K}{2}
\lambda. \end{split}
\end{align} 
We note that we used Lemma \ref{thetaest} (b) to obtain the third inequality.
Now from the definition of $ E(s_{1},K\lambda)$, we obtain
$$ \psi(\Theta, x)^{\frac{\nu}{nq}} \le  2^{(nq-1)\frac{\nu}{nq}} \psi_{k}^{+}(\Theta(u-h_{k}))^{\frac{\nu}{nq}}+\frac{1}{2}\psi(\Theta, x)^{\frac{\nu}{nq}} $$
 for any $x \in B_{5\tau_{k}}^{+}(y^{k}) \cap  E(s_{1},K\lambda)$
and this yields
$$  \psi(\Theta, x) \le 2^{\frac{nq}{\nu}+nq-1}\psi_{k}^{+}(\Theta(u-h_{k})).$$

Letting $\tilde{\theta}(t)=(\psi_{k}^{+})^{-1}(t^{\frac{nq}{\nu}})^{\nu}$,
$\tilde{\theta}$ satisfies $\textrm{(aInc)}_{1}$.
Then we can apply Lemma \ref{orljen} to $\tilde{\theta}$ and observe that
\begin{align} \label{ukji} \begin{split}
\bigg( (\psi_{k}^{+})^{-1} \bigg( & C\kint_{B_{5\tau_{k}}^{+}(y^{k})} \psi_{k}^{+}(\Theta(u-h_{k}))^{\frac{\nu}{nq}}  dx \bigg)^{\frac{nq}{\nu}} \bigg)^{\nu}
\\ &  \le \kint_{B_{5\tau_{k}}^{+}(y^{k})}
\big\{(\psi_{k}^{+})^{-1} \big( \psi_{k}^{+}(\Theta(u-h_{k})) \big) \big\}^{\nu} dx
\end{split}
\end{align}
for some $C=C(L)>0$, and by Lemma \ref{thetaest}, it follows that
\begin{align*}
\bigg( \kint_{B_{5\tau_{k}}^{+}(y^{k})} \psi_{k}^{+}(\Theta(u-h_{k}))^{\frac{\nu}{nq}}  dx \bigg)^{\frac{nq}{\nu}}
& \le C \psi_{k}^{+} \bigg(\bigg( \kint_{B_{5\tau_{k}}^{+}(y^{k})}\Theta(u-h_{k})^{\nu} dx \bigg)^{\frac{1}{\nu}} \bigg)
\\ & \le C  \psi_{k}^{+} \big (\epsilon(\psi_{k}^{+})^{-1}\big( \lambda^{\frac{nq}{\nu}} \big) \big)
\\ & \le C\epsilon^{np} \lambda^{\frac{nq}{\nu}}.
\end{align*}
Thus, we obtain
\begin{align}
\int_{B_{5\tau_{k}}^{+}(y^{k})}& \psi_{k}^{+}(\Theta(u-h_{k}))^{\frac{\nu}{nq}}  dx \le C\epsilon^{\frac{\nu p}{q}}
\lambda |B_{5\tau_{k}}^{+} |.
\end{align}

Therefore, we can deduce that
\begin{align*}
\int_{E(s_{1},K\lambda)}\psi(\Theta, x)^{\frac{\nu}{nq}}  dx & \le
C \sum_{k=1}^{\infty}\int_{B_{5\tau_{k}}^{+}(y^{k})} \psi_{k}^{+}(\Theta(u-h_{k}))^{\frac{\nu}{nq}}  dx
\\ & \le C \epsilon^{\frac{\nu p}{q}} \sum_{k=1}^{\infty} \lambda |B_{5\tau_{k}}^{+} |
\\ &  \le C \epsilon^{\frac{\nu p}{q}} \sum_{k=1}^{\infty} \lambda |B_{\tau_{k}}^{+} |
\end{align*}
for some $C=C(n,\lambda,\Lambda,q)$.

Then we can obtain
\begin{align}  \begin{split}
&\int_{E(s_{1},K\lambda)}\psi(\Theta, x)^{\frac{\nu}{nq}}  dx
\\ &\le C \epsilon^{\frac{\nu p}{q}} \sum_{k=1}^{\infty}\lambda  \bigg(
\frac{4}{\lambda} \int_{E_{k}(\lambda/4)}\psi(\Theta, x)^{\frac{\nu}{nq}} dx +
 \bigg( \frac{4}{\lambda\delta}\bigg)^{\frac{nq}{\nu p}}\int_{H_{k}(\lambda\delta/4)}\psi(|f|, x)^{\frac{1}{p}} dx \bigg)
\end{split}
\end{align}
by Lemma \ref{ehest}.
Since $\{ B_{\tau_{k}}^{+}(y^{k}) \}_{k=1}^{\infty}$ is mutually disjoint and $$ \bigcup_{k=1}^{\infty}B_{5\tau_{k}}^{+}(y^{k}) \subset B_{s_{2}}^{+},$$ we see that
\begin{align} \label{non2} \begin{split}
&\int_{E(s_{1},K\lambda)}\psi(\Theta, x)^{\frac{\nu}{nq}}  dx
\\ &\le C \epsilon^{\frac{\nu p}{q}} \lambda
 \bigg(\frac{4}{\lambda} \int_{E(s_{2},\lambda/4)}\psi(\Theta, x)^{\frac{\nu}{nq}} dx +
 \bigg( \frac{4}{\lambda\delta}\bigg)^{\frac{nq}{\nu p}}\int_{H(s_{2},\lambda\delta/4)}\psi(|f|, x)^{\frac{1}{p}} dx \bigg).
\end{split}
\end{align}

Define $\Xi(x)= \psi(\Theta, x)^{\frac{\nu}{nq}}$, $\Xi_{l}(x)=\min \{\Xi(x),l \}  $ for $l>0$ and $\chi = nq/\nu$.
Then we get
\begin{align*}
\int_{B_{s_{1}}^{+}}\Xi_{l}^{\chi-1} \Xi dx &
= (\chi -1)K^{\chi-1}\int_{0}^{\frac{l}{K}} \lambda^{\chi-2} \int_{E(s_{1},K\lambda)}  \Xi dx d\lambda
\\ & =(\chi -1)K^{\chi-1}\int_{0}^{\alpha \lambda_{0}} \lambda^{\chi-2} \int_{E(s_{1},K\lambda)}  \Xi dx d\lambda
\\ & \quad + (\chi -1)K^{\chi-1}\int_{\alpha \lambda_{0}}^{\frac{l}{K}} \lambda^{\chi-2} \int_{E(s_{1},K\lambda)}  \Xi dx d\lambda
\\ & =: I+J.
\end{align*}

For $I$, we have
\begin{align*}
|I| & \le  (\chi -1)K^{\chi-1} \bigg( \int_{0}^{\alpha \lambda_{0}} \lambda^{\chi-2}  d \lambda \bigg)
 \bigg( \int_{B_{s_{1}}^{+}} \Xi dx \bigg)
\\ & \le (K\alpha \lambda_{0})^{\chi-1}\int_{B_{2}^{+}} \Xi dx
\\ & \le \frac{C\lambda_{0}^{\chi-1}}{(s_{2}-s_{1})^{(\chi-1)nq}}\int_{B_{2}^{+}} \Xi dx,
\end{align*}
where $C=C(n,\lambda,\Lambda,p,q,L,t_{0})>0$.

On the other hand, for $J$, we can derive by using \eqref{non2} that
\begin{align*}
&|J|
\\ & \le C \epsilon^{\frac{\nu p}{q}} \lambda
\bigg| \int_{\alpha \lambda_{0}}^{\frac{l}{K}} \lambda^{\chi-2} \bigg\{
\frac{1}{\lambda} \int_{E(s_{2},\lambda/4)}\hspace{-2.5em}\psi(\Theta, x)^{\frac{\nu}{nq}} dx +
 \bigg( \frac{1}{\lambda\delta}\bigg)^{\frac{\chi}{p}}\hspace{-0.5em}
\int_{H(s_{2},\lambda\delta/4)}\hspace{-3.5em}\psi(|f|, x)^{\frac{1}{p}} dx
\bigg\}d\lambda \bigg|
\\ &  \le C \epsilon^{\frac{\nu p}{q}} \int_{\alpha \lambda_{0}}^{4l} \lambda^{\chi-2}  \int_{E(s_{2},\lambda/4)} \psi(\Theta, x)^{\frac{\nu}{nq}} dxd\lambda 
\\ & \qquad + C\epsilon^{\frac{\nu p}{q}} \int_{0}^{\infty} \lambda^{\chi-1-\frac{\chi}{p}}
\int_{H(s_{2},\lambda\delta/4)}\delta^{-\frac{\chi}{p}}\psi(|f|, x)^{\frac{1}{p}} dxd\lambda \\ &
\le  C \epsilon^{\frac{\nu p}{q}} \int_{B_{s_{2}}^{+}} \int_{0}^{\Xi_{l}(x)} \lambda^{\chi-2}\Xi dx d\lambda
\\ & \qquad +  C \epsilon^{\frac{\nu p}{q}} \int_{B_{2}^{+}} \int_{0}^{\psi(|f|, x)^{\frac{\nu}{nq}}} \lambda^{\chi-1-\frac{\chi}{p}}
\big( \delta^{-\chi}  \psi(|f|,x)\big)^{\frac{1}{p}} dx
\\ & \le C\epsilon^{\frac{\nu p}{q}} \int_{B_{s_{2}}^{+}}  \Xi_{l}^{\chi-1}\Xi dx + C\epsilon
\int_{B_{2}^{+}} \big( \psi(|f|,x)^{\frac{\nu}{nq}}\big)^{\chi-\frac{\chi}{p}} \big( \delta^{-\chi}  \psi(|f|,x)\big)^{\frac{1}{p}} dx
\\ & = C_{\ast} \epsilon^{\frac{\nu p}{q}} \bigg(  \int_{B_{s_{2}}^{+}}\Xi_{l}^{\chi-1}\Xi dx +
 \int_{B_{2}^{+}}  \delta^{-\chi}  \psi(|f|,x) dx \bigg).
\end{align*}

Now we choose $\epsilon$ small enough such that $C_{\ast}\epsilon^{\frac{\nu p}{q}}  = 1/2$
(we note that $\delta$ is also determined at this point).
Then
\begin{align*}
&\int_{B_{s_{1}}^{+}} \Xi_{l}^{\chi-1}\Xi dx \\ &
\le \frac{C\lambda_{0}^{\chi-1}}{(s_{2}-s_{1})^{(\chi-1)nq}}\int_{B_{2}^{+}}  \Xi dx +
\frac{1}{2} \int_{B_{s_{2}}^{+}} \Xi_{l}^{\chi-1}\Xi dx
+ C  \int_{B_{2}^{+}} \psi(|f|,x) dx .
\end{align*}
Let $\xi(s)=\int_{B_{s}^{+}}\Xi_{l}^{\chi-1}\Xi dx $ for $s>0$. Then we see that
\begin{align*}
\xi(s_{1}) \le \frac{C\lambda_{0}^{\chi-1}}{(s_{2}-s_{1})^{(\chi-1)nq}}\int_{B_{2}^{+}} \Xi dx
+ C \int_{B_{2}^{+}} \psi(|f|,x) dx +\frac{1}{2} \xi(s_{2}).
\end{align*}
By Lemma \ref{tchlm}, we obtain that
$$\xi(s_{1}) \le  C \bigg(  \frac{\lambda_{0}^{\chi-1}}{(s_{2}-s_{1})^{(\chi-1)nq}}\int_{B_{2}^{+}} \Xi dx+
\int_{B_{2}^{+}} \psi(|f|,x) dx  \bigg)$$
and this implies 
\begin{align*}
\int_{B_{1}^{+}} \Xi_{l}^{\chi-1}\Xi dx & \le
C \lambda_{0}^{\chi-1} \int_{B_{2}^{+}} \Xi dx+  C \int_{B_{2}^{+}}\psi(|f|,x) dx
\\ & \le C (n, \lambda, \Lambda, p,q,L,t_{0}).
\end{align*}
We have used the definition of $\lambda_{0}$ and the assumption $\kint_{B_{4}^{+}}\psi(|f|,x) dx \le 1$.
By applying the monotone convergence theorem,
\begin{align} \label{intlim}
\lim_{l \to \infty} \int_{B_{1}^{+}} \Xi_{l}^{\chi-1}\Xi dx = \int_{B_{1}^{+}}  \psi (\Theta,x)dx \le C,
\end{align}
where $C$ is a constant depending only on $n, \lambda, \Lambda, p,q, L$ and $t_{0}$.

\begin{proof}[Proof of Theorem \ref{main_ac}]
Let $$\kappa := ||f||_{L^{\psi(\cdot)}(B_{4}^{+})}
+||u||_{L^{\infty}(B_{4}^{+})}^n.$$
Now we consider $\tilde{F}(X,x)=\kappa^{-1}F(\kappa X,x)$, $\tilde{u}=\frac{\mu}{\kappa}u$,
$\tilde{\Theta}=\Theta (\tilde{u},B_{1}^{+})$ and $\tilde{f}=\frac{\mu}{\kappa}f$.
Note that we choose $\mu$ so that Lemma \ref{arr_ac} can be applied.
Then, $\tilde{u}$ is a viscosity solution of
\begin{align*}
\left\{ \begin{array}{ll}
\tilde{F}_{\mu}(D^{2}\tilde{u},x) = \tilde{f} & \textrm{in $B_{4}^{+}, $}\\
\tilde{u}=0 & \textrm{on $T_{4}$.}
\end{array} \right.
\end{align*}
We also observe that $||\tilde{u}||_{L^{\infty}(B_{4}^{+})} \le 1$ and
$\kint_{B_{4}^{+}} \psi(|\tilde{f}|,x)dx \le 1$.
Thus, by \eqref{intlim}, we deduce that
\begin{align*}
\int_{B_{1}^{+}} \psi (\tilde{\Theta},x) dx \le
C,
\end{align*}
where $C=C(n, \lambda, \Lambda, p,q,L,t_{0})$.
Now we finally obtain
\begin{align*}
||D^{2}u||_{L^{\psi(\cdot)}(B_{1}^{+})} \le C
( ||u||_{L^{\infty}(B_{4}^{+})}^n+||f||_{L^{\psi(\cdot)}(B_{4}^{+})})
\end{align*} from Lemma \ref{eqthehe} with unit ball property (\cite[Lemma 3.2.3]{MR3931352})
and then we can finish the proof.
\end{proof}

\section{Gradient and global estimates}
In this section, we first derive gradient estimates for a viscosity solution of
\begin{align}\label{groo}
\left\{ \begin{array}{ll}
F(D^{2}u,Du,u,x) = f & \textrm{in $B_{1}^{+}, $}\\
u=0 & \textrm{on $T_{1}$}
\end{array} \right.
\end{align}
 and then prove our desired result, Theorem \ref{gor_main}.

We begin this section with the following observation.
Let $u$ be a viscosity solution of \eqref{groo} and assume that it also solves
\begin{align}
\left\{ \begin{array}{ll}
F(D^{2}u,0,0,x) = \tilde{f} & \textrm{in $B_{1}^{+}, $}\\
u=0 & \textrm{on $T_{1}$}
\end{array} \right.
\end{align}
in the viscosity sense.
Then by \eqref{ob_sc}, we have
$$ |\tilde{f}| \le |f| + b|Du| + c|u| \quad \textrm{in}\ B_{1}^{+},$$
and this implies
$$||\tilde{f}||_{L^{\psi(\cdot)}(B_{1/4}^{+})} \le C( ||f||_{L^{\psi(\cdot)}(B_{1/2}^{+})}+
||u||_{W^{1,\psi(\cdot)}(B_{1/2}^{+})} )$$
for some $C=C(b,c)>0$.
We also note that a similar estimate can be obtained for the interior case.

Therefore, to obtain a Hessian estimate for $u$, we need to consider its gradient estimate.

\begin{theorem}\label{thm41}Assume that $\varphi$ is a weak $\Phi$-function satisfying (A0), (A1-$\varphi^{-}$), $\textrm{(aInc)}_{p}$ and $\textrm{(aDec)}_{q}$. Then the following hold:
\begin{itemize} 
\item[(a)] Suppose that $F=F(X,\xi,z,x)$ satisfies \eqref{ob_sc} in $B_{1}$, $f \in L^{\psi(\cdot)}(B_{1})$,
and there exists $F^{\star}$ satisfying \eqref{fstar} which is convex in $X$.
Then there exists a small constant $\delta$ depending on $n, \lambda, \Lambda, p,q,L$ and $t_{0}$  such that
$$\bigg( \kint_{B_{r}(x_{0})\cap B_{1}} \sigma_{F^{\star}}(x_{0},x)^{n} dx  \bigg)^{\frac{1}{n}} \le \delta $$
for any $r>0$ and $x_{0} \in B_{1}$ implies that for any viscosity solution $u$ of
$$ F(D^{2}u,Du,u,x) = f  \textrm{ in }B_{1},$$
we have $u \in W^{1, \psi(\cdot)}(B_{1/2})$ with the estimate
$$ ||u||_{ W^{1, \psi(\cdot)}(B_{1/2})} \le C(||u||_{ L^{\infty}(B_{1})}^n+||f||_{ L^{\psi(\cdot)}(B_{1})})$$
for some $C=C(n, \lambda, \Lambda, p,q,L,t_{0},b,c)>0$.
\item[(b)]
Suppose that $F=F(X,\xi,z,x)$ satisfies \eqref{ob_sc} in $B_{1}^{+}$, $f \in L^{\psi(\cdot)}(B_{1}^{+})$ and there exists $F^{\star}$ satisfying \eqref{fstar} which is convex in $X$.
Then there exists a small constant $\delta>0$ depending on $n, \lambda, \Lambda, p,q, L$
and $t_{0}$ such that
$$\bigg( \kint_{B_{r}(x_{0})\cap B_{1}^{+}} \sigma_{F^{\star}}(x_{0},x)^{n} dx  \bigg)^{\frac{1}{n}} \le \delta $$
for any $r>0$ and $x_{0} \in B_{1}^{+}$ implies that for any viscosity solution $u$ of
\begin{align*}
\left\{ \begin{array}{ll}
F(D^{2}u,Du,u,x) = f & \textrm{in $B_{1}^{+}, $}\\
u=0 & \textrm{on $T_{1},$}
\end{array} \right.
\end{align*}
we have $u \in W^{1, \psi(\cdot)}(B_{1/2}^{+})$ with the estimate
$$ ||u||_{ W^{1, \psi(\cdot)}(B_{1/2}^{+})} \le C(||u||_{ L^{\infty}(B_{1}^{+})}^n+||f||_{ L^{\psi(\cdot)}(B_{1}^{+})})$$
for some $C=C(n, \lambda, \Lambda, p,q,L,t_{0},b,c)>0$.
\end{itemize}
\end{theorem}

\begin{proof}  We only give the proof of (b) since (a) can also be derived in a similar way.  
Consider the following renormalization
$$ \tilde{u}(x)=\frac{u(x)}{||u||^n_{ L^{\infty}(B_{1}^{+})}+||f||_{L^{\psi(\cdot)}(B_{1}^{+})}  } .$$
Then we can assume that $||u||_{L^{\infty}(B_{1}^{+})},||f||_{ L^{\psi(\cdot)}(B_{1}^{+})} \le 1$.
Now we need to show that $$||Du||_{ L^{\psi(\cdot)}(B_{1/2}^{+})}  \le C$$
for some $C=C(n, \lambda, \Lambda, p,q,L,t_{0},b,c)>0$.

Let $p' \in (n, np)$, $d = \min \{ \frac{1}{2}, r_{0} \}$ and fix $y \in B_{1/2}^{+} \cup T_{1/2}$.
By using the similar argument in  \cite[Theorem 3.1]{MR2486925}, we can find an affine function $l$ such that
\begin{align*}
|l(0)|, |Dl(0)| \le CH(y) \quad \textrm{and}
\\ ||u-l||_{L^{\infty}(B_{r}(y)\cap \mathbb{R}_{+}^{n})} \le Cr^{1+\alpha}H(y)
\end{align*}
for every $y \in B_{1/2}^{+}$ and some $\alpha \in (0,1)$, $H$ is given by
\begin{align*}
H(y) = ||u||^n_{L^{\infty}(B_{d}^+(y))} + \beta_{0}^{-1}\sup_{r \le d}r^{1-\alpha}\bigg(
\kint_{B^+_{r}(y)} |f(y)|^{p'}\bigg)^{\frac{1}{p'}},
\end{align*} 
where $\beta_{0}>0$ is a small constant related to the constant $\delta>0$ in Lemma \ref{asyconapp} (for details on how to choose $\beta_0$, see, for example, the proof of \cite[Theorem 3.1]{MR2486925}).
Then we have
$$ \frac{|u(x+y)-u(y)|}{|x|} \le CH(y) $$
for almost every $y \in B_{1/2}^{+}$ and $x \in B_{\sigma}$ with small enough $\sigma$. 

Now we observe that
\begin{align*}
&\int_{B_{1/2}^{+}} \psi \bigg( \frac{|u(y+x)-u(y)|}{|x|},y\bigg) dy \\ & \le
\int_{B_{1/2}^{+}} \psi ( CH(y),y) dy
\\ & \le C(n,q)\int_{B_{1/2}^{+}} \psi ( H(y),y) dy
\\ & \le C(n,q)\int_{B_{1/2}^{+}} \hspace{-0.5em} \psi \bigg( ||u||^n_{L^{\infty}(B_{d}^+(y)} \hspace{-0.1em}+\hspace{-0.1em} \beta_{0}^{-1}\sup_{r \le d}r^{1-\alpha}\bigg(
\kint_{B_{d}^+(y)}\hspace{-1em} |f(x)|^{p'}dx\bigg)^{\frac{1}{p'}}\hspace{-0.5em} ,y \bigg) dy
\\ & \le C(n,q, t_{0})  \bigg\{ 1+ \int_{B_{1/2}^{+}}\psi \bigg(\beta_{0}^{-1}\sup_{r \le d}r^{1-\alpha}\bigg(
\kint_{B_{d}^+(y)}\hspace{-1em} |f(x)|^{p'}dx\bigg)^{\frac{1}{p'}},y \bigg) dy \bigg\}
\\ & \le C(n,q,t_{0},\zeta_{0})\int_{B_{1/2}^{+}} \psi \big( [\mathcal{M}(|f|^{p'})(x)]^{\frac{1}{p'}} , y \big) dy .
\end{align*}

Set $\tilde{\psi}(t,x) = \psi(t^{\frac{1}{p'}},x)$.
Then
\begin{align*}
\int_{B_{1/2}^{+}} \psi \big( [\mathcal{M}(|f|^{p'})(y)]^{\frac{1}{p'}} , y \big) dy=
\int_{B_{1/2}^{+}} \tilde{\psi} \big(  \mathcal{M}(|f|^{p'})(y), y \big) dy.
\end{align*}
By \cite[Theorem 4.3.4]{MR3931352}, the maximal operator $\mathcal{M}$ is bounded in $L^{\tilde{\psi}(\cdot)}$.
Hence, we obtain that
\begin{align*}
||\mathcal{M}(|f|^{p'})||_{L^{\tilde{\psi}(\cdot)}(B_{1/2}^{+})}
\le C |||f|^{p'}||_{L^{\tilde{\psi}(\cdot)}(B_{1/2}^{+})}
\le C ||f||_{L^{\psi(\cdot)}(B_{1/2}^{+})} \le C
\end{align*}
for some $C=C(n,p,q, t_{0}, L)$,
because we assumed that $||f||_{L^{\psi(\cdot)}(B_{1/2}^{+})} \le1$.
Applying \cite[Lemma 3.2.3]{MR3931352}, it follows that
$$ \int_{B_{1/2}^{+}} \tilde{\psi} (\mathcal{M}(|f|^{p'})(y),y) dy \le C.$$
Since $\tilde{\psi}(t,x) = \psi(t^{\frac{1}{p'}},x)$, we have
$$ \int_{B_{1/2}^{+}} \psi \big([\mathcal{M}(|f|^{p'})(y)]^{\frac{1}{p'}},y\big) dy \le C.$$
This completes the proof.
\end{proof}

Hence, using the observation at the beginning of this section together with Theorems \ref{main_ac} and Theorem \ref{thm41}, and exploiting the scaling property, we obtain the following Hessian estimates:
\begin{theorem} \label{thm42} Under the assumption of Theorem \ref{thm41}, the following hold:
\begin{itemize}
\item[(a)]Suppose that $F=F(X,\xi,z,x)$ satisfies  \eqref{ob_sc} in $B_{1}$, $f \in L^{\psi(\cdot)}(B_{1})$ and there exists $F^{\star}$ satisfying \eqref{fstar} which is convex in $X$.
Then there exists a small constant $\delta$ depending on $n, \lambda, \Lambda, p,q,L$ and $t_{0}$ such that
$$\bigg( \kint_{B_{r}(x_{0})\cap B_{1}} \sigma_{F^{\star}}(x_{0},x)^{n} dx  \bigg)^{\frac{1}{n}} \le \delta $$
for any $r>0$ and $x_{0} \in B_{1}$ implies that for any viscosity solution $u$ of
$$ F(D^{2}u,Du,u,x)=f \quad \textrm{in } B_{1},$$
we have $u \in W^{2, \psi(\cdot)}(B_{1/2})$ with the estimate
$$ ||u||_{ W^{2, \psi(\cdot)}(B_{1/2})} \le C(||u||_{ L^{\infty}(B_{1})}^n+||f||_{ L^{\psi(\cdot)}(B_{1})})$$
for some $C=C(n, \lambda, \Lambda, p,q,L,t_{0},b,c)>0$.
\item[(b)]
Suppose that $F=F(X,\xi,z,x)$ satisfies  \eqref{ob_sc} in $B_{1}^{+}$, $f \in L^{\psi(\cdot)}(B_{1}^{+})$ and there exists $F^{\star}$ satisfying \eqref{fstar} which is convex in $X$.
Then there exists a small constant $\delta>0$ depending on $n, \lambda, \Lambda, p,q,L$
and $t_{0}$ such that
$$\bigg( \kint_{B_{r}(x_{0})\cap B_{1}^{+}} \sigma_{F^{\star}}(x_{0},x)^{n} dx  \bigg)^{\frac{1}{n}} \le \delta $$
for any $r>0$ and $x_{0} \in B_{1}^{+}$ implies that for any viscosity solution $u$ of
\begin{align*}
\left\{ \begin{array}{ll}
F(D^{2}u,Du,u,x) = f & \textrm{in $B_{1}^{+}, $}\\
u=0 & \textrm{on $T_{1},$}
\end{array} \right.
\end{align*}
we have $u \in W^{2, \psi(\cdot)}(B_{1/2}^{+})$ with the estimate
$$ ||u||_{ W^{2, \psi(\cdot)}(B_{1/2}^{+})} \le C(||u||_{ L^{\infty}(B_{1}^{+})}^n+||f||_{ L^{\psi(\cdot)}(B_{1}^{+})})$$
for some $C=C(n, \lambda, \Lambda, p,q, L,t_{0},b,c)>0$.
\end{itemize}
\end{theorem}

Now we can give the proof of our main theorem.
\begin{proof}[Proof of Theorem \ref{gor_main}] 

Since we assumed $ \partial \Omega $ has the $C^{1,1}-$regularity, there exists a neighborhood $U(x_{0}) $ and a $ C^{1,1}$-diffeomorphism
$$ \Psi : U(x_{0}) \to B_{1}  \cap \mathbb{R}_{+}^{n} $$
such that $ \Psi(x_{0}) = 0$  for each $ x_{0} \in \partial \Omega $.
Set  $ \tilde{u} =  u \circ \Psi^{-1}$. Then $\tilde{u} $ is a viscosity solution of
 \begin{align*}
\left\{ \begin{array}{ll}
\tilde{F}(D^{2}\tilde{u},D\tilde{u},\tilde{u},x)=\tilde{f} & \textrm{in $B_{1}^{+}$,}\\
\tilde{u} = 0 & \textrm{on $T_{1}$,}\\
\end{array} \right.
\end{align*}
where $ \tilde{f} = f \circ \Psi^{-1}$ and
\begin{align*}
\tilde{F}(D^{2}\tilde{\phi},D\tilde{\phi},\tilde{u},x) &
=F (D^{2}\phi, D\phi, u ,x) \circ \Psi^{-1} \\ &
= F(D \Psi^{T} \circ \Psi^{-1} D^{2} \tilde{\phi} D\Psi \circ \Psi^{-1} + (D \tilde{\phi} \partial_{i,j} \Psi \circ \Psi^{-1})_{1\le i,j \le n} , \\ & \qquad  \qquad
D\phi D \Psi \circ \Psi^{-1}, \tilde{u} , \Psi^{-1}(x))
\end{align*}
for $ \tilde{\phi} \in W^{2,p}(B_{1}^{+} )$ and $ \phi = \tilde{\phi} \circ \Psi ( \in W^{2,p}(U(x_{0})))$.
We can check that $\tilde{F}$ is uniformly elliptic with constants $C_{0}\lambda  $, $C_{0} \Lambda $
for a constant $ C_{0}= C_{0}(\Psi)$, and
$$ \sigma_{\tilde{F}}(x,x_{0}) \le C(\Psi) \sigma_{F}(\Psi^{-1}(x),\Psi^{-1}(x_{0}))$$
(see \cite{MR2486925}).
By using a covering argument with Theorem \ref{thm42}, we obtain that a viscosity solution $u$ of \eqref{ob_acori} satisfies
$$ ||u||_{ W^{2, \psi(\cdot)}(\Omega)} \le C(||u||_{ L^{\infty}(\Omega)}^n+||f||_{ L^{\psi(\cdot)}(\Omega)})$$
for some $C=C(n, \lambda, \Lambda, p,q ,L,t_{0},b,c)>0 $.

Now we need to eliminate the term $||u||_{ L^{\infty}(\Omega)}^n $ in the above estimate.
To this end, we use a similar argument to the proof of \cite[Theorem 2.6]{MR3390010}.
Suppose that the estimate \eqref{est_main} is not true.
Then we can choose sequences $\{u_k\}_{k=1}^{\infty}$ and $\{f_k\}_{k=1}^{\infty}$ so that
$u_k$ is a viscosity solution to
\begin{align*} 
\left\{ \begin{array}{ll}
F(D^{2}u_k,Du_k,u_k,x) = f_k & \textrm{in $\Omega, $}\\
u_k=0 & \textrm{on $\partial \Omega$}
\end{array} \right.
\end{align*}
with
\begin{align}\label{adco1}
    ||u_k||_{ W^{2, \psi(\cdot)}(\Omega)} > k||f_k||_{ L^{ \psi(\cdot)}(\Omega)} 
\end{align}
for any $k\ge1$.
Assume that $||u_k||_{ W^{2, \psi(\cdot)}(\Omega)}=1 $ without loss of generality.
Then it follows that
\begin{align}\label{adco2}
  ||f_k||_{ L^{ \psi(\cdot)}(\Omega)} <   \frac{1}{k}
\end{align}
and the right-hand side term tends to zero as $k\to\infty$.
On the other hand, we also see that $u_k$ converges weakly to a function $v\in  W^{2, \psi(\cdot)}(\Omega)$ in $W^{2, \psi(\cdot)}(\Omega) $ by passing to a subsequence.

Now observe that $W^{2, \psi(\cdot)}(\Omega)\subset W^{2, \psi^+}(\Omega) $.
Since $  W^{2, \psi^+}(\Omega)$ is continuously embedded in $W^{2,nq}(\Omega)$ and 
$W^{2,nq}(\Omega) $ is compactly embedded in $C(\Omega)$ due to $nq>\frac{n}{2}$,
we obtain that  $W^{2, \psi(\cdot)}(\Omega) $ is also compactly embedded in $C(\Omega)$.
This yields that $u_k$ converges strongly to $v$ in $C(\Omega)$.
In addition, by \cite[Proposition 1.5]{MR2486925}, $v$ is a viscosity solution of
\begin{align}\label{homo}
\left\{ \begin{array}{ll}
F(D^{2}v,Dv,v,x) = 0 & \textrm{in $\Omega, $}\\
v=0 & \textrm{on $\partial \Omega.$}
\end{array} \right.
\end{align}
Then, from the uniqueness of solutions to \eqref{homo} (see \cite[Theorem 2.10]{MR1376656}), we have $v\equiv 0$.
It follows that
$$1= ||u_k||_{ W^{2, \psi(\cdot)}(\Omega)}\le C(||u_k||_{ L^{\infty}(\Omega)}^n+||f_k||_{ L^{\psi(\cdot)}(\Omega)}),$$
and the right-hand side term converges to zero.
But this is a contradiction, and hence we can complete the proof.
\end{proof}


\end{document}